\documentclass[final,11pt]{article}
\usepackage{amsthm}
\usepackage{cite}
\usepackage{amsmath,amssymb,amsfonts}
\usepackage[dvipsnames,usenames]{color}

\usepackage{hyperref}
\usepackage[bitstream-charter]{mathdesign}
\usepackage[utf8]{inputenc}
\usepackage{graphicx}
\usepackage{tikz}
\usepackage{caption}
\usepackage{subcaption}
\usepackage{enumerate}
\usepackage{geometry}
\usepackage[affil-it]{authblk}
\usepackage{tikz-cd}
\usepackage{titling}

\newtheorem{thm}{Theorem}[section]
\newtheorem{prop}[thm]{Proposition}
\newtheorem{lem}[thm]{Lemma}
\newtheorem{defi}[thm]{Definition}

\newtheorem{cor}[thm]{Corollary}
\newtheorem{claim}[thm]{Claim}
\newtheorem{obs}[thm]{Observation}
\newtheorem{conj}{Conjecture}

\newcommand{\conv}{\ensuremath{{\rm conv}}}

\newcommand{\R}{\ensuremath{\mathbb{R}}}

\renewcommand{\Im}{\ensuremath{{\rm Im}}}

\renewcommand{\H}{\ensuremath{{\mathcal{H}}}}

\newcommand{\conn}{\ensuremath{{\rm conn}}}
\newcommand{\supp}{\ensuremath{{\rm supp}}}
\renewcommand{\int}{\ensuremath{{\rm int}}}
\newcommand{\Inc}{\ensuremath{{\rm Inc}}}
\newcommand{\ch}{\ensuremath{{\rm ch}}}
\newcommand{\low}{\ensuremath{{\rm low}}}

\begin{document}
\thanksmarkseries{arabic}
\title{Edge-Isoperimetric Inequalities in Chamber Graphs of Hyperplane Arrangements}
\author{Tilen Marc\thanks{Electronic address: \texttt{tilen.marc@fmf.uni-lj.si}}}
\affil{Faculty of Mathematics and Physics, Ljubljana, Slovenia \\
    Institute of Mathematics, Physics, and Mechanics, Ljubljana, Slovenia}

\sloppy
\maketitle

\begin{abstract}
    We study edge-isoperimetric inequalities in chamber graphs of affine hyperplane arrangements.
    Our approach is topological: to a set of chambers we associate its thickening in Euclidean space
    and estimate its edge boundary through the induced stratification by intersections of arrangement
    hyperplanes. This yields general lower bounds for a broad class of sets.
    We show that a convex set of chambers of size
    $\sum_{i=0}^d \binom{k}{i}$, with $k\ge d-1$, has edge boundary at least $\sum_{i=0}^{d-1}\binom{k}{i}$, and we
    conjecture that convex sets minimize the edge boundary among all chamber sets of a fixed size.
    We verify this conjecture in dimension $2$. Our main result is a three-dimensional asymptotic
    inequality for arbitrary subsets of chambers: for arrangements in general position, every set
    $S$ occupying at most a fixed proportion of the chambers satisfies
    $|\partial S|=\Omega(|S|^{2/3})$. As a consequence, for an arrangement of $n$ hyperplanes in
    general position in $\mathbb R^3$, the lazy simple random walk on the chamber graph has
    $\varepsilon$-mixing time $O(n^2\log(n/\varepsilon))$.
\end{abstract}

\section{Introduction}
Let $\mathcal{H} = \{H_1,\dots,H_n\}$ be a finite arrangement of affine hyperplanes in
$\mathbb{R}^d$.
The complement
\[
    \mathbb{R}^d \setminus \bigcup_{i=1}^n H_i
\]
decomposes into finitely many open connected components, each of which is a convex polyhedron,
possibly unbounded.
These components are called the \emph{chambers} (or \emph{regions}) of the arrangement.

The \emph{chamber graph} of $\mathcal{H}$ is the graph whose vertices are the chambers, where two
chambers are adjacent if they share a common facet (equivalently, if they are separated by exactly
one hyperplane). In the context of oriented matroids, this graph is known as the \emph{tope graph}.

For a graph $G=(V,E)$ and a subset $S\subseteq V$, the \emph{edge boundary} of $S$ is defined as
\[
    \partial S := \{\, xy\in E \mid x\in S,\ y\in V\setminus S \,\}.
\]
In this paper we study isoperimetric properties of chamber graphs, with the aim of establishing
lower bounds on $|\partial S|$ in terms of $|S|$.

\subsection{Motivation}

The edge-isoperimetric problem is a classical topic in graph theory. For highly structured graphs such as
the hypercube and the grid, sharp inequalities are known; see, for example,
\cite{Harper1964,BollobasLeader1991Grid,BollobasLeader1991Compressions}. Their significance extends
well beyond the size of the boundary. Lower bounds on $|\partial S|$ control the edge-expansion, or
isoperimetric number, of the graph, and these quantities are linked by Cheeger-type inequalities to
conductance and to the bottom of the graph Laplacian spectrum
\cite{Chung1997,LawlerSokal1988}. In particular, isoperimetric inequalities yield lower bounds on the
spectral gap and hence upper bounds on relaxation and mixing times for lazy random walks.

These consequences are algorithmically important. When the vertices of a graph represent feasible
solutions to a combinatorial problem, rapid mixing gives almost-uniform sampling; for self-reducible
families, such sampling can then be converted into approximate counting
\cite{JerrumValiantVazirani1986,JerrumSinclair1989}. This perspective already appears for subsets of
the hypercube defined by linear inequalities: Morris and Sinclair analyzed random walks on truncated
cubes and obtained a rapidly mixing chain for sampling $0$--$1$ knapsack solutions
\cite{MorrisSinclair2004}.

Chamber graphs of hyperplane arrangements provide a natural setting in which to extend this circle of
ideas. They retain substantial geometric structure---half-spaces, convex regions, and the
stratification induced by intersections of hyperplanes---but, unlike cubes and grids, they generally
lack a global product structure. As a result, the usual compression arguments are unavailable, and one
is led to ask whether comparable isoperimetric inequalities continue to hold in this broader geometric
setting. One of the outcomes of the present work is that, in dimension $3$, such an inequality also
yields polynomial mixing bounds for the lazy nearest-neighbor random walk on the chamber graph.

\subsection{Related Work}

The best-developed edge-isoperimetric theory concerns product graphs. The hypercube case goes back to
Harper \cite{Harper1964}, while Bollob\'as and Leader established sharp and near-sharp inequalities for
grids and related compression frameworks \cite{BollobasLeader1991Grid,BollobasLeader1991Compressions}.
These results are a natural benchmark for the present work, but their proofs rely heavily on
coordinate-wise compression and other features that are special to product settings.

Direct isoperimetric inequalities for chamber graphs seem to be largely absent from the literature.
By contrast, several closely related spectral and mixing questions have been studied for Markov chains
arising from hyperplane arrangements. A particularly elegant example is the Bidigare--Hanlon--Rockmore
walk analyzed by Brown and Diaconis \cite{BrownDiaconis1998}: one chooses a face $F$ and moves from a
chamber $C$ to the projected chamber $FC$, namely the nearest chamber having $F$ as a face. This
theory gives stationary distributions, diagonalizability, and general convergence bounds for the
resulting chains. However, the chain is not, in general, the nearest-neighbor random walk on the
chamber graph, since a single step may cross several facets at once.

Finite Coxeter groups provide a different and more graph-theoretic class of examples. If one takes the
standard Coxeter generators, the Cayley graph is exactly the chamber graph of the corresponding
reflection arrangement: chambers are indexed by group elements, and crossing one reflecting hyperplane
corresponds to multiplication by a simple reflection. In this setting Kassabov showed that for every
finite Coxeter group the Laplacian spectral gap with respect to the standard generators can be computed
from the defining representation \cite[Theorem~6.1]{Kassabov2011}; in particular, in type $A_n$ exact
spectral-gap values are obtained \cite[Example~6.2]{Kassabov2011}. Since these chamber graphs are
regular, such spectral information in turn yields bounds on their isoperimetric number through
Cheeger-type inequalities. Thus spectral questions have been studied on chamber graphs in our sense,
even though direct edge-isoperimetric inequalities seem to be missing.

Fulman used the Brown--Diaconis hyperplane-walk viewpoint to study riffle shuffles on finite Coxeter
groups \cite{Fulman2001}, and a different but related line of work studies expansion in basis-exchange
graphs of matroids. Feder and Mihail introduced balanced matroids and proved strong expansion properties
for their exchange graphs \cite{FederMihail1992}, while Anari, Liu, Oveis Gharan, and Vinzant proved
the Mihail--Vazirani conjecture that every matroid basis graph has edge expansion at least $1$
\cite{AnariLiuOveisGharanVinzant2024}. These results concern different graph families, but they further
illustrate the central role of isoperimetry and expansion in combinatorial geometry.

On the oriented matroid side, chamber graphs appear as tope graphs; standard references are the
monograph \cite{OrientedMatroids1999} and the later work on complexes of oriented matroids and their
tope graphs \cite{BandeltChepoiKnauer2018,KnauerMarc2020}. These papers are primarily concerned with
structural, metric, and convexity properties of tope graphs rather than with quantitative lower bounds
for edge boundary.

To the best of our knowledge, general isoperimetric inequalities for chamber graphs of hyperplane
arrangements have so far been available only in much more special situations, such as product-type
graphs. The present paper is a first step toward such inequalities for general hyperplane arrangements
and their tope graphs.

\subsection{Approach and main results}

Our approach is topological. Instead of working directly with the chamber graph, we study the
stratification induced by a set $S$ of chambers and bound $|\partial S|$ through the topological
boundaries of the corresponding strata. This yields a general framework, but the point where the
argument becomes difficult is the passage from topologically simple sets to arbitrary ones.

For topologically simple sets, the stratification bounds are already strong enough to give general
isoperimetric estimates. This applies in particular to convex sets, but the method is not restricted
to convexity alone. In this setting we obtain explicit lower
bounds on the edge boundary; see Proposition~\ref{prop:num_edges} and
Corollary~\ref{cor:convex_bound}. These results suggest that convex sets should be extremal. We therefore conjecture that, among sets
of chambers of a fixed size, convex sets minimize the edge boundary.

The same stratification
approach also leads to general results in low dimensions. In dimension $2$ we prove the conjectured
bound. The main achievement, however, is in dimension $3$: there the method goes beyond the
topologically simple setting and yields a nontrivial general bound for arbitrary subsets. More
precisely, we prove an asymptotic form of the conjecture by showing that
$|\partial S|=\Omega(|S|^{2/3})$ for sets that are bounded away from the full arrangement; see
Theorem~\ref{thm:R3_asymp}.
The same three-dimensional analysis also has a random-walk application: for an arrangement
of $n$ hyperplanes in general position in $\R^3$, it implies that the lazy simple random walk on the
chamber graph has $\varepsilon$--mixing time $O(n^2\log(n/\varepsilon))$; see
Corollary~\ref{cor:R3_mixing}.

\section{Preliminaries and basic lemmas}

Let $\mathcal{H}$ be a hyperplane arrangement in $\mathbb{R}^d$ and let $S$ be a subset of its chambers.
We shall view $S$ both as a set of vertices in the chamber graph and as a subset of $\mathbb{R}^d$
via the following construction.

\begin{defi}\label{def:topological}
    Let $\mathcal{H} = \{H_e \mid e\in[n]\}$ be a hyperplane arrangement in $\mathbb{R}^d$, and let
    $S$ be a subset of its chambers.
    \begin{itemize}
        \item The \emph{regular open hull} (or \emph{thickening}) of $S$ is
              \[
                  T_S := \operatorname{int}\!\left(\overline{\bigcup_{C\in S} C}\right).
              \]
              Equivalently, a point $x\in\mathbb{R}^d$ belongs to $T_S$ if either $x$ lies in a chamber of $S$,
              or all chambers whose closures contain $x$ belong to $S$.

        \item For $e\in[n]$, we say that the hyperplane $H_e$ \emph{intersects} $S$ if
              $T_S\cap H_e\neq\emptyset$.

        \item For a subset $A\subseteq[n]$, define the affine subspace
              \[
                  L_A := \bigcap_{f\in A} H_f ,
              \]
              and the corresponding \emph{stratum}
              \[
                  T_S(A) := T_S \cap L_A .
              \]
              The set
              \[
                  \supp(S) := \{\, A\subseteq[n] \mid T_S(A)\neq\emptyset \,\}
              \]
              is called the \emph{support} of $S$.

        \item For $A\in\supp(S)$, denote by $\conn(T_S(A))$ the set of connected components of $T_S(A)$.
    \end{itemize}
\end{defi}

To analyze the edge boundary $\partial S$, we will relate it to the geometry of the boundary
of $T_S$. For a hyperplane $H_e$, denote by $H_e^+$ and $H_e^-$ the two open half-spaces determined
by $H_e$.

\begin{defi}\label{def:boundary}
    Let $\mathcal{H}$ be a hyperplane arrangement in $\mathbb{R}^d$ and let $S$ be a subset of its chambers.
    \begin{itemize}
        \item Let
              \[
                  \Sigma_S := \partial T_S
              \]
              denote the topological boundary of $T_S$.

        \item For $A\in\supp(S)$ and $e\in[n]$, we say that the hyperplane $H_e$ \emph{bounds}
              the stratum $T_S(A)$ if
              \begin{itemize}
                  \item $T_S(A)$ is contained entirely in one of the half-spaces $H_e^+$ or $H_e^-$, and
                  \item the closure of $T_S(A)$ meets $H_e$ within $L_A$, i.e.,
                        \[
                            H_e \cap \overline{T_S(A)} \cap L_A \neq \emptyset .
                        \]
              \end{itemize}

        \item For $A\in\supp(S)$, let $b(A)$ denote the set of all hyperplanes in $\mathcal{H}$ that
              bound $T_S(A)$.
    \end{itemize}
\end{defi}

Figure~\ref{fig:defs_example} illustrates Definitions~\ref{def:topological}
and~\ref{def:boundary} with an arrangement of three lines in~$\mathbb{R}^2$.
The three lines $H_1,H_2,H_3$ are in general position and form one bounded
chamber (the triangle) and six unbounded chambers.  We take $S$ to consist of
the triangle together with the two adjacent chambers that share an edge with it
on $H_1$ and $H_2$, respectively.  Since every chamber in $S$ lies on the same
side of $H_3$, the hyperplane $H_3$ does not intersect $S$, and hence
$\{3\}\notin\supp(S)$.  The support is
$\supp(S)=\{\emptyset,\{1\},\{2\}\}$: the stratum $T_S(\{1\})$ is the open
segment $T_S\cap H_1$ and $T_S(\{2\})$ is the open segment $T_S\cap H_2$.
On the boundary side, $b(\emptyset)=\{H_3\}$,
$b(\{1\})=\{H_2,H_3\}$, and $b(\{2\})=\{H_1,H_3\}$.

\begin{figure}[ht]
    \centering
    \begin{subfigure}[b]{0.47\textwidth}
        \centering
        \begin{tikzpicture}[scale=0.82]
            \fill[cyan!10] (-1,1) -- (-2.2,2.2) -- (2.2,2.2) -- (1,1) -- cycle;
            \fill[cyan!10] (0,0) -- (1,1) -- (2.5,1) -- (2.5,-1.9) -- (1.9,-1.9) -- cycle;
            \fill[cyan!10] (-1,1) -- (1,1) -- (0,0) -- cycle;
            \draw[thick] (-2.5,1)--(2.5,1) node[below left] {$H_1$};
            \draw[thick] (-1.9,-1.9)--(2.2,2.2) node[above right] {$H_2$};
            \draw[thick] (1.9,-1.9)--(-2.2,2.2) node[above left] {$H_3$};
            \draw[ultra thick, orange!85!black] (-0.96,1)--(0.96,1);
            \draw[ultra thick, green!55!black] (0.04,0.04)--(0.96,0.96);
            \node[cyan!55!black] at (-0.15,0.55) {$S$};
            \node[orange!85!black] at (0,1.3) {\small $T_S(\{1\})$};
            \node[green!55!black] at (1.1,0.2) {\small $T_S(\{2\})$};
        \end{tikzpicture}
        \caption{The three chambers in $S$ (shaded), the strata
            $T_S(\{1\})=T_S\cap H_1$ and $T_S(\{2\})=T_S\cap H_2$.}
    \end{subfigure}
    \hspace{0.03\textwidth}
    \begin{subfigure}[b]{0.47\textwidth}
        \centering
        \begin{tikzpicture}[scale=0.82]
            \fill[cyan!8] (-1,1) -- (-2.2,2.2) -- (2.2,2.2) -- (1,1) -- cycle;
            \fill[cyan!8] (0,0) -- (1,1) -- (2.5,1) -- (2.5,-1.9) -- (1.9,-1.9) -- cycle;
            \fill[cyan!8] (-1,1) -- (1,1) -- (0,0) -- cycle;
            \draw[thick] (-2.5,1)--(2.5,1) node[below left] {$H_1$};
            \draw[thick] (-1.9,-1.9)--(2.2,2.2) node[above right] {$H_2$};
            \draw[thick] (1.9,-1.9)--(-2.2,2.2) node[above left] {$H_3$};
            \draw[ultra thick, red!75!black] (1.9,-1.9)--(-2.2,2.2);
            \draw[ultra thick, red!75!black] (1,1)--(2.2,2.2);
            \draw[ultra thick, red!75!black] (1,1)--(2.5,1);
            \node[red!75!black] at (-1.1,0.15) {$\Sigma_S$};
            \node[orange!75!black, font=\small] at (0.15,-1.1) {$H_3\!\in\! b(\emptyset)$};
            \fill[orange] (1,1) circle (3pt);
            \node[orange!85!black, font=\small, anchor=south west] at (1.55,0.95) {$H_1\!\in\! b(\{2\})$};
            \node[orange!85!black, font=\small, anchor=south west] at (-0.9,1.15) {$H_2\!\in\! b(\{1\})$};
            \fill[orange] (-1,1) circle (3pt);
            \node[orange!85!black, font=\small, anchor=north east] at (-0.8,1.0) {$H_3\!\in\! b(\{1\})$};
            \fill[orange] (0,0) circle (3pt);
            \node[orange!85!black, font=\small, anchor=north west] at (0.2,0.35) {$H_3\!\in\! b(\{2\})$};
        \end{tikzpicture}
        \caption{The boundary $\Sigma_S=\partial T_S$ (thick red) and bounding hyperplanes.}
    \end{subfigure}
    \caption{Three lines in $\mathbb{R}^2$ forming one bounded and six unbounded
        chambers.  $S$ consists of the triangle and two adjacent chambers,
        giving $\supp(S)=\{\emptyset,\{1\},\{2\}\}$.}
    \label{fig:defs_example}
\end{figure}
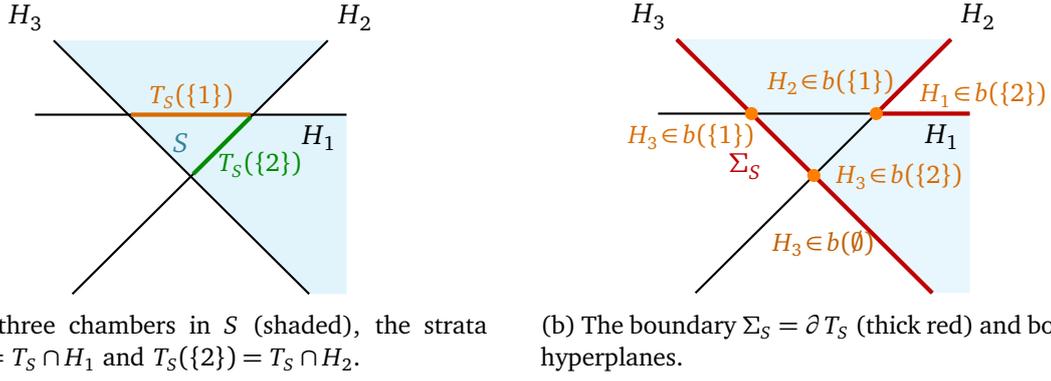


Throughout the paper we assume that $\mathcal{H}$ is in \emph{general position}, meaning that
the intersection of any $k\le d$ hyperplanes is an affine subspace of dimension $d-k$, and that
no $d+1$ hyperplanes have a common intersection.

The following observations are immediate consequences of general position.

\begin{obs}
    Let $\mathcal{H}$ be a hyperplane arrangement in general position in $\mathbb{R}^d$.
    \begin{itemize}
        \item For every $A\in\supp(S)$, the stratum $T_S(A)$ is a $(d-|A|)$-dimensional subset of $L_A$.
        \item If $C\in\conn(T_S(A))$ and $C'\in\conn(T_S(A'))$ with $A\neq A'$, then $C\neq C'$.
    \end{itemize}
\end{obs}

We now record a structural lemma describing how strata behave under cutting by a hyperplane. It
will be useful for proving the main results of the paper.

\begin{lem}\label{lem:strata_conn}
    Let $S$ be a subset of chambers of a hyperplane arrangement $\mathcal{H}$ in general position
    in $\mathbb{R}^d$, and let $e\in[n]$.
    Define
    \[
        S^+ := \{\, C\in S \mid C\subset H_e^+ \,\}.
    \]
    Let $A\in\supp(S)$ with $e\notin A$.
    If $c\in\conn(T_S(A\cup\{e\}))$, then $c$ lies on the boundary of exactly one connected component
    of $T_{S^+}(A)$.
\end{lem}
\begin{proof}
    Let $c\in\conn(T_S(A\cup\{e\}))$.
    Since
    \[
        T_S(A\cup\{e\}) = T_S(A)\cap H_e ,
    \]
    the set $c$ is contained in $T_S(A)$.
    In particular, there exists at least one connected component of $T_{S^+}(A)$
    whose boundary contains $c$.

    Consider the hyperplane arrangement induced by $\mathcal{H}$ on the affine
    subspace
    \[
        L_{A\cup\{e\}} = L_A \cap H_e .
    \]
    The set $c$ is the topological closure of a union of chambers of this induced arrangement.
    Let $U_c$ denote the set of all such chambers.
    Since $c$ is connected, the subgraph of the chamber graph of the induced
    arrangement on $L_{A\cup\{e\}}$ induced by $U_c$ is connected.

    Let $R_1$ and $R_2$ be two adjacent chambers in $U_c$.
    Because the arrangement is in general position, $R_1$ and $R_2$ are separated
    by exactly one hyperplane.
    Each of $R_1$ and $R_2$ is a codimension-one face of a unique chamber of the
    induced arrangement on $L_A$ contained in the half-space $H_e^+$.
    Denote these chambers by $R_1'$ and $R_2'$.
    Since $R_1$ and $R_2$ are separated by a single hyperplane, the chambers
    $R_1'$ and $R_2'$ are adjacent in the chamber graph of the induced arrangement
    on $L_A$.

    It follows that all chambers of $T_{S^+}(A)$ whose boundaries contain faces in
    $U_c$ lie in the same connected component of $T_{S^+}(A)$.
    Therefore, $c$ lies entirely on the boundary of exactly one connected component
    of $T_{S^+}(A)$.
\end{proof}

\begin{lem}\label{lem:gluing_graph}
    Let $\mathcal H$ be in general position and let $S$ be a set of chambers.
    Fix $e\in[n]$ and write
    \[
        S^+ := \{C\in S : C\subset H_e^+\},\qquad
        S^- := \{C\in S : C\subset H_e^-\}.
    \]
    Fix $A\subset[n]$ with $e\notin A$ and assume $A\in\supp(S)$.

    Define a bipartite graph $\Gamma_{A,e}$ as follows:
    \begin{itemize}
        \item Vertex set:
              \[
                  V(\Gamma_{A,e}) := \conn(T_{S^+}(A))\ \sqcup\ \conn(T_{S^-}(A)).
              \]
        \item Edge set:
              \[
                  E(\Gamma_{A,e}) := \conn(T_S(A\cup\{e\})).
              \]
        \item Incidence: for $D\in\conn(T_S(A\cup\{e\}))$, connect the unique component
              $C^+\in\conn(T_{S^+}(A))$ whose boundary contains $D$ to the unique
              $C^-\in\conn(T_{S^-}(A))$ whose boundary contains $D$ (existence/uniqueness given by
              Lemma~\ref{lem:strata_conn}).
    \end{itemize}

    Then:
    \begin{enumerate}[(i)]
        \item The connected components of $\Gamma_{A,e}$ are in bijection with $\conn(T_S(A))$.
              In particular,
              \[
                  c(\Gamma_{A,e}) = |\conn(T_S(A))|.
              \]
        \item Consequently,
              \[
                  |\conn(T_{S^+}(A))| + |\conn(T_{S^-}(A))|
                  \ \le\
                  |\conn(T_S(A))| + |\conn(T_S(A\cup\{e\}))|.
              \]
    \end{enumerate}
\end{lem}

\begin{proof}
    (i) Consider the set
    \[
        X := T_{S^+}(A)\ \cup\ T_{S^-}(A)\ \cup\ T_S(A\cup\{e\}) \subseteq L_A .
    \]
    By construction, $X = T_S(A)$, since crossing $H_e$ inside $L_A$ is recorded exactly by
    the stratum $T_S(A\cup\{e\}) = T_S(A)\cap H_e$.
    Moreover, Lemma~\ref{lem:strata_conn} (and its symmetric version for $S^-$) implies that each
    component $D\in \conn(T_S(A\cup\{e\}))$ lies on the boundary of exactly one component of
    $T_{S^+}(A)$ and exactly one component of $T_{S^-}(A)$, so the adjacency relation encoded by
    $\Gamma_{A,e}$ precisely describes how the components of $T_{S^+}(A)$ and $T_{S^-}(A)$ are glued
    together through $T_S(A\cup\{e\})$. Therefore two vertices of $\Gamma_{A,e}$ lie in the same
    connected component of $\Gamma_{A,e}$ if and only if the corresponding subsets lie in the same
    connected component of $X=T_S(A)$. This gives the bijection and hence
    $c(\Gamma_{A,e}) = |\conn(T_S(A))|$.

    (ii) Let $c := c(\Gamma_{A,e})$. Each connected component of $\Gamma_{A,e}$ contains a spanning tree,
    so in that component we have $|V_i|-1 \le |E_i|$. Summing over all components gives
    \[
        |V(\Gamma_{A,e})| - c \le |E(\Gamma_{A,e})|.
    \]
    Substitute $|V(\Gamma_{A,e})| = |\conn(T_{S^+}(A))| + |\conn(T_{S^-}(A))|$,
    $|E(\Gamma_{A,e})| = |\conn(T_S(A\cup\{e\}))|$, and $c=|\conn(T_S(A))|$ from (i), yielding (ii).
\end{proof}

We will repeatedly use the following version of the well-known Kruskal--Katona theorem; see
\cite{Kruskal1963,Katona1968}.

\begin{thm}[Kruskal--Katona]\label{thm:kruskal_katona}
    For real $x$ and integer $i\ge 0$, define
    \[
        \binom{x}{i}:=\frac{x(x-1)\cdots(x-i+1)}{i!}.
    \]
    Let $m\ge 0$, let $x\ge m-1$, and let $\mathcal D\subseteq 2^{[n]}$ be a down-set, that is, $A\in\mathcal D$
    and $B\subseteq A$ imply $B\in\mathcal D$. For $0\le i\le m$, write
    \[
        \mathcal D^{(i)}:=\{A\in\mathcal D\mid |A|=i\},
    \]
    and assume that $\mathcal D^{(i)}=\emptyset$ for $i>m$. Then:
    \begin{enumerate}[(i)]
        \item If
              \[
                  |\mathcal D|\le \sum_{i=0}^m \binom{x}{i},
              \]
              then
              \[
                  |\mathcal D^{(m)}|\le \binom{x}{m}.
              \]
        \item More generally, if $0\le r\le m$ and
              \[
                  \sum_{i=r}^m |\mathcal D^{(i)}|=\sum_{i=r}^m \binom{x}{i},
              \]
              then for every sequence
              \[
                  w_r\ge w_{r+1}\ge \cdots \ge w_m\ge 0
              \]
              one has
              \[
                  \sum_{i=r}^m w_i\,|\mathcal D^{(i)}|
                  \ge
                  \sum_{i=r}^m w_i\,\binom{x}{i}.
              \]
    \end{enumerate}
\end{thm}

\section{First results and a conjecture}

Our approach to proving lower bounds on the edge boundary is based
on the stratification of the set. Our first result bounds the
number of vertices (chambers) in terms of the strata.

\begin{prop}\label{prop:num_vertices}
    Let $S$ be a subset of chambers of a hyperplane arrangement $\mathcal{H}$ in
    general position in $\mathbb{R}^d$.
    There exists an injective map
    \[
        \mu : S \longrightarrow \bigsqcup_{A\in\supp(S)} \conn(T_S(A))
    \]
    such that for every $A\in\supp(S)$ at least one connected component of $T_S(A)$
    is hit by $\mu$, i.e.,
    \[
        \mu(S) \cap \conn(T_S(A)) \neq \emptyset .
    \]
    Consequently,
    \[
        |\supp(S)| \;\le\; |S| \;\le\; \sum_{A\in\supp(S)} |\conn(T_S(A))| .
    \]
\end{prop}
\begin{proof}
    We proceed by induction on the number of hyperplanes intersecting $T_S$.
    Without loss of generality we can assume that $T_S$ is connected.

    We start with the base case.
    If no hyperplane intersects $T_S$, then $T_S$ consists of a single chamber.
    Hence $\supp(S)=\{\emptyset\}$ and
    \[
        |S| = |\conn(T_S(\emptyset))|.
    \]
    In this case the statement is trivial.

    Now we continue with the inductive step.
    Assume that some hyperplane $H_e$ intersects $T_S$.
    Define
    \[
        S^+ := \{\, C\in S \mid C\subset H_e^+ \,\},
        \qquad
        S^- := \{\, C\in S \mid C\subset H_e^- \,\}.
    \]
    Then $S=S^+\sqcup S^-$ and $|S|=|S^+|+|S^-|$.

    By induction, for each $\sigma\in\{+,-\}$ there exists an injective map
    \[
        \mu_\sigma : S^\sigma \to \bigsqcup_{A\in\supp(S^\sigma)} \conn(T_{S^\sigma}(A))
    \]
    such that for every $A\in\supp(S^\sigma)$, at least one component of
    $\conn(T_{S^\sigma}(A))$ is hit.

    We now construct an injective map
    \[
        \mu : S \to \bigsqcup_{A\in\supp(S)} \conn(T_S(A)).
    \]

    Fix $A\in\supp(S)$ with $e\notin A$.

    \smallskip
    \noindent
    \emph{Case 1: $A\cup\{e\}\notin\supp(S)$.}
    Then no gluing occurs along $H_e$, and
    \[
        \conn(T_S(A)) = \conn(T_{S^+}(A)) \;\sqcup\; \conn(T_{S^-}(A)).
    \]
    We define $\mu$ on chambers mapped by $\mu_+$ and $\mu_-$ to these components
    by setting $\mu=\mu_+$ on $S^+$ and $\mu=\mu_-$ on $S^-$.

    \smallskip
    \noindent
    \emph{Case 2: $A\cup\{e\}\in\supp(S)$.}
    In this case, components of $T_{S^+}(A)$ and $T_{S^-}(A)$ may be glued together
    along components of $T_S(A\cup\{e\})$.
    By Lemma~\ref{lem:gluing_graph}(ii),
    \[
        |\conn(T_{S^+}(A))| + |\conn(T_{S^-}(A))|
        \ \le\
        |\conn(T_S(A))| + |\conn(T_S(A\cup\{e\}))|.
    \]

    By the inductive hypothesis, at least one chamber of $S^+$ and one of $S^-$
    is mapped by $\mu_+$ and $\mu_-$ to components of
    $\conn(T_{S^+}(A))$ and $\conn(T_{S^-}(A))$, respectively.
    Using the inequality above, we may reassign the images of all chambers
    mapped into these components so that:
    \begin{itemize}
        \item the reassignment is injective,
        \item at least one chamber maps to a component of $\conn(T_S(A))$,
        \item at least one chamber maps to a component of $\conn(T_S(A\cup\{e\}))$.
    \end{itemize}

    Performing this construction independently for all $A\in\supp(S)$
    defines an injective map $\mu : S \to \bigsqcup_{A\in\supp(S)} \conn(T_S(A))$
    such that every $A\in\supp(S)$ is hit.
    This completes the induction.
\end{proof}

If one encodes chambers by sign vectors, then the sets
$A\in\supp(S)$ are precisely the strongly shattered subsets of the corresponding
sign-vector family: equivalently, $S$ contains a subcube of the sign hypercube
with free coordinates $A$; see, for example,
\cite{ChalopinChepoiMoranWarmuth2022,MoranRashtchian2016}. Thus the lower bound
\[
    |\supp(S)| \le |S|
\]
from Proposition~\ref{prop:num_vertices} is the lower half of the classical
Sandwich Lemma, which generalizes the Sauer--Shelah--Perles lemma. By contrast,
the upper bound
\[
    |S| \le \sum_{A\in\supp(S)} |\conn(T_S(A))|
\]
is of a different nature and should not be confused with the number of
shattered subsets appearing in the usual upper half of the Sandwich Lemma.

We now turn to lower bounds on the edge boundary $\partial S$. We first
show that small sets have large boundary.

\begin{prop}\label{prop:small}
    Let $S$ be a subset of chambers of a hyperplane arrangement $\mathcal{H}$ in
    general position in $\mathbb{R}^d$, $|\mathcal{H}| \ge d$.
    If $|S|\le 2^{d-1}$, then
    \[
        |\partial S| \;\ge\; |S|.
    \]
\end{prop}
\begin{proof}
    Let $G$ be the chamber graph of $\mathcal H$, let $n:=|\mathcal H|$, and put $m:=|S|$.
    To each chamber $C$ we associate its sign vector
    \[
        \sigma(C)\in\{\pm1\}^{n},
    \]
    whose $e$th coordinate records on which side of the hyperplane $H_e$ the chamber $C$
    lies. Two chambers are adjacent in $G$ if and only if they are separated by exactly one
    hyperplane, so the map $C\mapsto \sigma(C)$ embeds $G$ isometrically into the
    $n$--dimensional hypercube $Q_n$. In particular, $G[S]$ is isomorphic to an induced
    subgraph of $Q_n$ on $m$ vertices.

    Write
    \[
        e(S):=|E(G[S])|.
    \]
    By Harper's edge-isoperimetric inequality for the hypercube
    \cite{Harper1964} (see also \cite{BollobasLeader1991Compressions,RashtchianRaynaud2022}),
    every set $A\subseteq V(Q_n)$ of size $m$ satisfies
    \[
        |\partial_{Q_n} A|\ge m\bigl(n-\log_2 m\bigr).
    \]
    Applying this to $A=\sigma(S)$, we obtain
    \[
        nm \;=\; 2e(S)+|\partial_{Q_n}\sigma(S)|
        \;\ge\; 2e(S)+m\bigl(n-\log_2 m\bigr),
    \]
    and hence
    \begin{equation}\label{eq:internal-edge-bound-small}
        e(S)\le \frac12\,m\log_2 m.
    \end{equation}

    Every chamber of $\mathcal H$ has degree at least $d$ in $G$. Now we count boundary edges:
    \[
        |\partial S|
        =
        \sum_{C\in S}\deg_G(C)-2e(S)
        \ge
        dm-2e(S).
    \]
    Using \eqref{eq:internal-edge-bound-small}, we obtain
    \[
        |\partial S|
        \ge
        dm-m\log_2 m
        =
        m\bigl(d-\log_2 m\bigr).
    \]
    Finally, since $m\le 2^{d-1}$, we have $\log_2 m\le d-1$, and therefore
    \[
        |\partial S|
        \ge
        m\bigl(d-\log_2 m\bigr)\ge m = |S|,
    \]
    as claimed.
\end{proof}

For the following result we shall limit ourselves to topologically simple $S$:

\begin{defi}\label{def:strata_connected}
    A subset $S$ of chambers is called \emph{strata-connected} if for every
    $A\in\supp(S)$ the stratum $T_S(A)$ is connected.
\end{defi}
If $S$ is strata-connected, then
the lower and upper bounds from Proposition~\ref{prop:num_vertices} coincide.
We give a lemma that will be used in the proof of the main result of this section, Proposition~\ref{prop:num_edges}.


\begin{lem}\label{lem:bound_lifts_from_section}
    Let $S$ be a strata-connected subset of chambers
    of a hyperplane arrangement $\mathcal{H}$ and let $H_e$ intersect $T_S$.
    Let $A\in\supp(S)$ with $e\in A$, and write $A':=A\setminus\{e\}$.
    If $H_f\in b_S(A)$, then either:
    \begin{itemize}
        \item $H_f$ bounds both
              $T_{S^+}(A')$ and $T_{S^-}(A')$, or
        \item $H_f$ bounds one of
              $T_{S^+}(A')$ or $T_{S^-}(A')$ and intersects the other.
    \end{itemize}
\end{lem}

\begin{proof}
    Since $e\in A$, the stratum $T_S(A)$ is contained in $H_e$ and in the affine space
    $L_A:=\bigcap_{i\in A} H_i = H_e\cap L_{A'}$, where $L_{A'}:=\bigcap_{i\in A'} H_i$.
    The assumption $H_f\in b_S(A)$ means that $T_S(A)\subset H_f^{\sigma}$ for some sign $\sigma\in\{+,-\}$,
    and that the relative boundary of $T_S(A)$ in $L_A$ meets $H_f$.

    Now consider $T_S(A')$ as a subset of $L_{A'}$.
    For each $\tau\in\{+,-\}$, the set $T_S(A)$ lies in the boundary of $T_{S^\tau}(A')$, so
    $H_f$ meets the boundary of $T_{S^\tau}(A')$. If $T_{S^\tau}(A')\subset H_f^{\sigma}$, then by
    definition $H_f\in b_{S^\tau}(A')$. If this is not the case, $H_f$ intersects $T_{S^\tau}(A')$.

    To conclude the proof, we must prove that $H_f$ cannot intersect both $T_{S^+}(A')$ and
    $T_{S^-}(A')$. If this is the case, $A' \cup \{f\} \in \supp(S^+)$ and
    $A' \cup \{f\} \in \supp(S^-)$. This implies that $A' \cup \{f\} \in \supp(S)$ with
    $|\conn(T_S(A' \cup \{f\}))| \geq 2$, since the two connected components in $T_{S^+}(A'\cup \{f\})$ and
    $T_{S^-}(A'\cup \{f\})$ are not connected by $T_{S}(A \cup \{f\})$. In fact,
    $A \cup \{f\} \notin \supp(S)$, since $H_f \in b_S(A)$. This contradicts the assumption that $S$ is strata-connected.
\end{proof}

\begin{prop}\label{prop:num_edges}
    Let $S$ be a strata-connected subset of the chambers of a hyperplane arrangement
    $\mathcal{H}$ in general position in $\mathbb{R}^d$.
    Then
    \[
        |\partial S| \;\ge\; \sum_{A\in\supp(S)} |b(A)|.
    \]
\end{prop}

\begin{proof}
    We proceed by induction on the number of hyperplanes intersecting $T_S$.

    We start with the base case.
    If no hyperplane intersects $T_S$, then $S$ consists of a single chamber.
    In this case $\supp(S)=\{\emptyset\}$ and $|\partial S|=|b(\emptyset)|$, so the inequality holds.

    We now continue with the inductive step.
    Assume that $H_e$ intersects $T_S$.
    As before, for each $\sigma\in\{+,-\}$ define
    \[
        S^\sigma := \{ C\in S \mid C\subset H_e^\sigma\}.
    \]
    Let $S/e$ denote the set of edges in the chamber graph of $\mathcal{H}$
    connecting $S^+$ and $S^-$. Equivalently, $S/e$ corresponds to the set of chambers
    of the induced arrangement on $H_e$ corresponding
    to the strata $T_S(\{e\})$.

    The boundary of $S$ decomposes as
    \begin{equation}\label{eq:boundary_decomposition}
        |\partial S| \;=\; |\partial S^+| + |\partial S^-| - 2|S/e|.
    \end{equation}

    Since $S$ is strata-connected, Lemma~\ref{lem:strata_conn} implies that each stratum
    of $S$ splits into at most two strata under restriction to the two open half-spaces $H_e^+$
    and $H_e^-$.
    Consequently, $S^+$ and $S^-$ are also strata-connected.
    Moreover, $S/e$ is strata-connected as a subset of chambers of the affine arrangement on $H_e$.

    By the inductive hypothesis,
    \[
        |\partial S^\sigma| \;\ge\; \sum_{A\in\supp(S^\sigma)} |b_{S^\sigma}(A)|
        \qquad\text{for each }\sigma\in\{+,-\}.
    \]
    By Proposition~\ref{prop:num_vertices}, strata-connectedness of $S/e$ implies
    \[
        |S/e| \;=\; |\supp(S/e)|.
    \]
    Substituting these inequalities into \eqref{eq:boundary_decomposition}, it suffices to prove
    \begin{equation}
        \sum_{A\in\supp(S^+)} |b_{S^+}(A)| + \sum_{A\in\supp(S^-)} |b_{S^-}(A)|
        \;\ge\;
        \sum_{A\in\supp(S)} |b_S(A)| + 2|\supp(S/e)|.
    \end{equation}

    Define the disjoint unions
    \[
        L \;:=\;
        \bigsqcup_{A\in\supp(S^+)} \bigl(b_{S^+}(A)\times\{+\}\bigr)
        \;\sqcup\;
        \bigsqcup_{A\in\supp(S^-)} \bigl(b_{S^-}(A)\times\{-\}\bigr),
    \]
    and
    \[
        R \;:=\;
        \bigsqcup_{A\in\supp(S)} b_S(A) \times \{A \}
        \;\sqcup\;
        \bigl(\supp(S/e)\times\{+,-\}\bigr).
    \]
    By construction,
    \[
        |L| \;=\; \sum_{A\in\supp(S^+)} |b_{S^+}(A)| + \sum_{A\in\supp(S^-)} |b_{S^-}(A)|,
        \quad
        |R| \;=\; \sum_{A\in\supp(S)} |b_S(A)| + 2|\supp(S/e)|.
    \]
    To prove the desired inequality, we now define an injective map $\gamma:R\to L$.

    In the first step we consider terms in $R$ that do not involve $e$.
    Let $H_f\in b_S(A)$ with $e\notin A$.
    \begin{enumerate}[(a)]
        \item
              If $T_S(A)\subset H_e^+$ (resp.\ $H_e^-$), then $T_{S^+}(A)=T_S(A)$
              (resp.\ $T_{S^-}(A)=T_S(A)$) and $H_f$ remains bounding on that side.
              Define $\gamma(H_f,A):=(H_f,A,+)$ (resp.\ $(H_f,A,-)$).

        \item
              If $T_S(A)$ meets both $H_e^+$ and $H_e^-$, then $T_S(A)$ decomposes as
              \[
                  T_S(A) = T_{S^+}(A)\,\cup\,T_{S^-}(A)\,\cup\,T_S(A\cup\{e\}),
              \]
              and $H_f$ bounds at least one of $T_{S^+}(A)$ or $T_{S^-}(A)$.
              If $H_f$ bounds both of them, define $\gamma(H_f,A):=(H_f,A,+)$. If it
              bounds exactly one of them, define  $\gamma(H_f,A):=(H_f,A,\sigma)$,
              where $\sigma$ is the corresponding sign.
    \end{enumerate}

    In the second step we consider terms in $R$ that involve $e$.
    For each $A\in\supp(S/e)$, the hyperplane $H_e$ bounds both
    $T_{S^+}(A\setminus\{e\})$ and $T_{S^-}(A\setminus\{e\})$.
    Define
    \[
        \gamma(A,+) := (H_e, A\setminus\{e\}, +),
        \qquad
        \gamma(A,-) := (H_e, A\setminus\{e\}, -).
    \]

    Finally consider $H_f$ that bounds $T_{S}(A)$ for $A\in\supp(S)$, $e\in A$. Define
    $A' = A \setminus \{e\}$. By
    Lemma~\ref{lem:bound_lifts_from_section}, we have two cases. If
    $H_f$ bounds both $T_{S^+}(A')$ and $T_{S^-}(A')$, define
    $\gamma(H_f,A):=(H_f,A',-)$. On the other hand, if it
    bounds exactly one of them, define $\gamma(H_f,A):=(H_f,A',\sigma)$,
    where $\sigma$ is the corresponding sign. Note that in the latter
    case $H_f$ does not bound $T_S(A')$, since it intersects one of
    $T_{S^+}(A')$ or $T_{S^-}(A')$.

    To conclude the proof, we check that $\gamma$ is injective.
    Write elements of $L$ as triples $(H_f,A,\sigma)$, meaning $H_f\in b_{S^\sigma}(A)$, where
    $\sigma\in\{+,-\}$. We shall prove that for a triple $(f,A,\sigma) \in \Im(\gamma)$,
    one can reconstruct the domain element mapping to it uniquely:
    \begin{itemize}
        \item If $f=e$, then $e \not\in A$, and the pre-image is uniquely determined as $(A \cup \{e\},\sigma) \in \supp(S/e)\times\{+,-\}$.
        \item If $f\neq e$, then the pre-image is of the form $(H_f,A')$ for some $A'\in\supp(S), H_f \in b_S(A')$,
              where $A'$ is either $A$ or $A\cup\{e\}$:
              \begin{itemize}
                  \item If $\sigma=+$, then the pre-image is $(H_f,A)$.
                  \item If $\sigma=-$, then the pre-image depends on whether
                        $A\cup\{e\}\in\supp(S)$ and $H_f \in b_S(A\cup\{e\})$ or not. If the
                        former holds, the pre-image is $(H_f,A\cup\{e\})$.
                        In the latter case, $A\in\supp(S)$ and the pre-image is $(H_f,A)$.
              \end{itemize}
    \end{itemize}
\end{proof}

We can now use Proposition~\ref{prop:num_edges} to obtain explicit lower bounds for
natural classes of subsets of chambers.
In particular, one expects that for a fixed cardinality $|S|$ the smallest edge boundary
should be attained by ``geometrically simple'' sets, and the most basic such class is
given by convex subsets.

A subset $S$ of the chambers of an arrangement $\mathcal{H}$ is called \emph{convex}
if the corresponding thickening $T_S$ is a convex open subset of $\mathbb{R}^d$.
Equivalently, $S$ is convex in the graph-theoretic sense: for any two chambers
$C_1,C_2\in S$, every shortest path between $C_1$ and $C_2$ in the chamber graph
lies entirely in $S$.
The equivalence follows from the fact that chamber graphs of affine hyperplane arrangements
are partial cubes (i.e.\ isometric subgraphs of hypercubes), where convex vertex sets
are precisely intersections of halfspaces; see, for example,
\cite{BjornerZiegler1992} or \cite{BandeltChepoi2008}.
We say that $S$ is \emph{proper} if it is not the set of all chambers.

We have already shown that small sets have large boundary in Proposition~\ref{prop:small},
so we focus on sets of size at least $2^{d-1}$.
For such sets, we can write $|S|=\sum_{i=0}^d \binom{k}{i}$ for some $k\ge d -1$.

\begin{cor}\label{cor:convex_bound}
    Let $S$ be a proper convex subset of the chambers of a hyperplane arrangement
    $\mathcal{H}$ in general position in $\mathbb{R}^d$.
    Assume that
    \[
        |S|=\sum_{i=0}^d \binom{k}{i}
        \qquad\text{for some }k\ge d - 1.
    \]
    Then
    \[
        |\partial S| \;\ge\; \sum_{i=0}^{d-1} \binom{k}{i}.
    \]
    If moreover $T_S$ is bounded in $\mathbb{R}^d$, then
    \[
        |\partial S| \;\ge\; \sum_{i=0}^{d-1} (d-i+1)\binom{k}{i}.
    \]
    If $|S| \leq \sum_{i=0}^d \binom{n-d}{i}$, where $n:=|\mathcal{H}|$, then
    \[
        |\partial S| \;\ge\; \sum_{i=0}^{d-1} (d-i)\binom{k}{i}.
    \]
    The latter always holds if $|S|$ is at most half of the total number of chambers
    and $n\ge \frac{d^2}{\ln 2}+2d$.
\end{cor}
\begin{proof}
    Since $S$ is convex and proper, the thickening $T_S$ is a proper convex polyhedron (possibly unbounded) in
    $\mathbb{R}^d$. In particular, $T_S$ is an intersection
    of open halfspaces bounded by hyperplanes of $\mathcal{H}$:
    \begin{equation}\label{eq:Ts_halfspace_representation}
        T_S \;=\; \bigcap_{j=1}^{\ell} H_{e_j}^{\sigma_j},
        \qquad \sigma_j\in\{+,-\},
    \end{equation}
    where each $H_{e_j}\in\mathcal{H}$ and the halfspaces are chosen so that $T_S$ lies
    strictly on the corresponding side.

    We first prove that convexity implies that $S$ is strata-connected.
    Fix $A\in\supp(S)$. The stratum
    \[
        T_S(A)=T_S\cap \bigcap_{e\in A} H_e
    \]
    is the intersection of the convex set $T_S$ with the affine subspace
    $L_A:=\bigcap_{e\in A} H_e$.
    Hence $T_S(A)$ is convex (and open in the relative topology of $L_A$), and therefore
    connected whenever it is nonempty. Thus $|\conn(T_S(A))|=1$ for all $A\in\supp(S)$,
    i.e.\ $S$ is strata-connected in the sense of Def~\ref{def:strata_connected}.
    By Proposition~\ref{prop:num_vertices} we obtain
    \begin{equation}\label{eq:S_equals_support}
        |S| \;=\; |\supp(S)|.
    \end{equation}

    In the second step we establish lower bounds on $|b(A)|$.
    Let $A\in\supp(S)$ with $|A|<d$ and set $r:=\dim(L_A)=d-|A|$.
    Then $T_S(A)$ is a nonempty convex open subset of the $r$-dimensional affine space $L_A$.
    Since $S$ is a proper convex subset of chambers and the arrangement is in general position,
    $T_S(A)\neq L_A$ for all $|A|<d$.
    Using the representation \eqref{eq:Ts_halfspace_representation}, we have
    \[
        T_S(A) \;=\; L_A \cap \bigcap_{j=1}^{\ell} H_{e_j}^{\sigma_j}.
    \]
    Since $T_S(A)\neq L_A$, at least one of the halfspace constraints is active on $L_A$,
    i.e.\ there exists $j$ such that $L_A\cap H_{e_j}$ meets the relative boundary of $T_S(A)$ in $L_A$.
    Therefore $H_{e_j}$ bounds $T_S(A)$ in the sense of Definition~\ref{def:boundary}, and thus
    \begin{equation}\label{eq:bA_at_least_1}
        |b(A)|\ge 1
        \qquad\text{for all }A\in\supp(S)\text{ with }|A|<d.
    \end{equation}

    Assume now that $T_S$ is bounded in $\mathbb{R}^d$.
    Then for every $A\in\supp(S)$ with $|A|<d$, the stratum $T_S(A)\subseteq L_A$ is bounded in $L_A$.
    We claim that in this case
    \begin{equation}\label{eq:bA_bounded_lower}
        |b(A)| \;\ge\; r+1 \;=\; d-|A|+1.
    \end{equation}
    Indeed, $T_S(A)$ is a bounded open convex polyhedron in $L_A$, defined by the halfspaces
    $H_{e_j}^{\sigma_j}\cap L_A$. A bounded convex polyhedron in an $r$-dimensional affine space
    must have at least $r+1$ facets, and each facet is supported by a distinct bounding hyperplane.
    Consequently $T_S(A)$ is bounded by at least $r+1$ hyperplanes from $\mathcal{H}$, proving
    \eqref{eq:bA_bounded_lower}.

    By Proposition~\ref{prop:num_edges},
    \[
        |\partial S| \;\ge\; \sum_{A\in\supp(S)} |b(A)|.
    \]
    Using \eqref{eq:bA_at_least_1}, we obtain
    \begin{equation}\label{eq:convex_reduction_1}
        |\partial S|
        \;\ge\;
        \sum_{\substack{A\in\supp(S)\\ |A|<d}} 1
        \;=\;
        |\{A\in\supp(S)\mid |A|<d\}|.
    \end{equation}

    To bound the right-hand side in terms of $|S|$, note that $\supp(S)$ is a down-set:
    if $A\in\supp(S)$ and $B\subseteq A$, then
    $T_S(A)\subseteq T_S(B)$, so $T_S(B)\neq\emptyset$ and hence $B\in\supp(S)$.
    For $0\le i\le d$, write $\supp(S)^{(i)}:=\{A\in\supp(S)\mid |A|=i\}$.

    By \eqref{eq:S_equals_support} and the assumption on $|S|$ we have
    \[
        |\supp(S)| \;=\; |S|
        \;=\; \sum_{i=0}^d \binom{k}{i}.
    \]
    for some $k\ge d - 1$.
    Applying Theorem~\ref{thm:kruskal_katona}(ii) to the down-set $\supp(S)$ with $r=0$, $m=d$,
    and weights $w_i=1$ for $0\le i\le d-1$ and $w_d=0$, we obtain
    \[
        \sum_{i=0}^{d-1} |\supp(S)^{(i)}|
        \;\ge\;
        \sum_{i=0}^{d-1} \binom{k}{i}.
    \]
    Therefore
    \[
        |\{A\in\supp(S)\mid |A|<d\}|
        \;=\;
        \sum_{i=0}^{d-1} |\supp(S)^{(i)}|
        \;\ge\;
        \sum_{i=0}^{d-1} \binom{k}{i}.
    \]
    Combining with \eqref{eq:convex_reduction_1} proves the first inequality.

    Now assume $T_S$ is bounded. Then \eqref{eq:bA_bounded_lower} yields
    \[
        |\partial S|
        \;\ge\;
        \sum_{\substack{A\in\supp(S)\\ |A|<d}} (d-|A|+1)
        \;=\;
        \sum_{i=0}^{d-1} (d-i+1)\,|\supp(S)^{(i)}|.
    \]
    Applying Theorem~\ref{thm:kruskal_katona}(ii) with weights $w_i=d-i+1$ for $0\le i\le d-1$
    and $w_d=0$ gives
    \[
        |\partial S|
        \;\ge\;
        \sum_{i=0}^{d-1} (d-i+1)\binom{k}{i},
    \]
    as claimed.

    Finally, assume that
    $|S| \leq \sum_{i=0}^d \binom{n-d}{i}$. First assume that the boundary of $T_S$ consists of $b \geq d$ hyperplanes,
    as expressed in equation \ref{eq:Ts_halfspace_representation}.
    Then for every $A\in\supp(S)$ with $|A|<d$, the stratum $T_S(A)$ is a nonempty convex open subset of the affine space $L_A$,
    bounded by at least $d-|A|$ hyperplanes from $\mathcal{H}$.
    Therefore
    \[
        |\partial S|
        \;\ge\;
        \sum_{\substack{A\in\supp(S)\\ |A|<d}} (d-|A|)
        \;=\;
        \sum_{i=0}^{d-1} (d-i)\,|\supp(S)^{(i)}|.
    \]
    Applying Theorem~\ref{thm:kruskal_katona}(ii) with weights $w_i=d-i$ for $0\le i\le d-1$
    and $w_d=0$ gives
    \[
        |\partial S|
        \;\ge\;
        \sum_{i=0}^{d-1} (d-i)\binom{k}{i}.
    \]

    Now assume that $b < d$ and write $\mathcal{H}_b$ for the set of bounding hyperplanes.
    Since the arrangement is in general position,
    any $L_A := \bigcap_{f\in A} H_f$ for $A \subset \mathcal{H} \setminus \mathcal{H}_b$, $|A| \leq d-b$,
    intersects $T_S$. For each such $A$, $T_S(A)$ is bounded in $L_A$ by
    all of the hyperplanes $\mathcal{H}_b$. On the other hand, if $A\in\supp(S)$ with $|A| > d-b$,
    then $|b(A)| \geq d-|A|$.
    As above we can write
    \[
        |\supp(S)| \;=\; |S|
        \;=\; \sum_{i=0}^d |\supp(S)^{(i)}|,
    \]
    where now we have $|\supp(S)^{(i)}| = {n - b \choose i}$ for $0\leq i \leq d-b$.
    On the other hand,
    \[
        |\partial S|
        \;\ge\;
        \sum_{\substack{A\in\supp(S)\\ |A|<d}} \min(d-|A|, b)
        \;=\;
        \sum_{i=0}^{d-b} b\,|\supp(S)^{(i)}| + \sum_{i=d-b+1}^{d-1} (d-i)\,|\supp(S)^{(i)}|.
    \]

    Let us write $\sum_{i=d-b+1}^d |\supp(S)^{(i)}| = \sum_{i=d-b+1}^d \binom{k_b}{i}$, for some $k_b\geq d-1$.
    Applying Theorem~\ref{thm:kruskal_katona}(ii) to the down-set $\supp(S)$ with $r=d-b+1$, $m=d$,
    and weights $w_i=d-i$ for $d-b+1\le i\le d-1$ and $w_d=0$, we get
    \[
        |\partial S|
        \;\ge\;
        \sum_{i=0}^{d-b} b\,{\binom{n-b}{i}} + \sum_{i=d-b+1}^{d-1} (d-i)\,\binom{k_b}{i}.
    \]

    By Lemma~\ref{prop:binomial_inequality}, the larger the value of $b$, the weaker
    the lower bound. In particular, for $b=d$ we get
    \[
        |\partial S|
        \;\ge\;
        \sum_{i=0}^{d-1} (d-i)\,\binom{k}{i},
    \]
    proving the last inequality in the statement of the corollary. Finally, the condition
    $n\ge \frac{d^2}{\ln 2}+2d$ ensures, by Lemma~\ref{lem:half_chambers_threshold}, that
    $\sum_{i=0}^{d}\binom{n-d}{i}\;\ge\;\frac{1}{2}\sum_{i=0}^{d}\binom{n}{i}$. If $|S|$ is
    at most half of the total number of chambers, then $|S| \leq \sum_{i=0}^d \binom{n-d}{i}$, and the last inequality applies.
\end{proof}

Note that the crucial property of convex sets used in the proof above is that
their strata are connected and each is bounded by at least one hyperplane.
Therefore, the same argument can also be applied to other sets $S$ with this
property. For example, consider a hyperplane arrangement in $\R^d$ in which
each hyperplane is orthogonal to one of the coordinate axes. Then the chamber
graph is isomorphic to a grid graph. This is, in fact, one of the few chamber
graphs that has already been studied in the literature; see~\cite{BollobasLeader1991Grid,BollobasLeader1991Compressions}.
For the grid graph, the first step is to ``push down'' a set $S$ along each
of the $d$ coordinate directions. This operation does not increase the size of
$\partial S$ but makes the set topologically simple. Then, for every
$A\in\supp(S)$, the stratum $T_S(A)$ is, in each of the $d-|A|$ directions,
either bounded by a hyperplane or equal to the whole intersection
$\cap_{e\in A} H_e$. Applying Proposition~\ref{prop:num_edges}, one can derive
asymptotically similar bounds to those given in~\cite{BollobasLeader1991Grid,BollobasLeader1991Compressions}. Unfortunately, it is
not clear how a ``push-down'' operation should be defined, or whether it can be
defined at all, for an arbitrary hyperplane arrangement.

It is counterintuitive that sets with disconnected strata
would have fewer boundary edges than convex sets. Hence we propose
the following conjecture.

\begin{conj}\label{conj:1}
    Let $S$ be an arbitrary subset of at most half of the chambers of a hyperplane arrangement $\mathcal{H}$ in
    general position in $\R^d$. Let $|S| = \sum_{i=0}^d {k \choose  i}$
    for some $k \ge d-1$. Then
    \[
        |\partial S| \geq \sum_{i=0}^{d-1} {k \choose  i}.
    \]
\end{conj}

One could also formulate a weaker conjecture concerning only the asymptotic size of the boundary.
In particular, if the above conjecture is true, then for a set $S$ of chambers in
$\R^d$ with fixed $d$, it follows that $|\partial S|$ is in $\Omega(|S|^\frac{d-1}{d})$.
We prove Conjecture~\ref{conj:1} in $\R^2$, and our main low-dimensional result is that the
asymptotic conjecture holds in $\R^3$.

\begin{figure}[ht]
    \centering
    \begin{subfigure}[b]{0.38\textwidth}
        \centering
        \includegraphics[width=\textwidth]{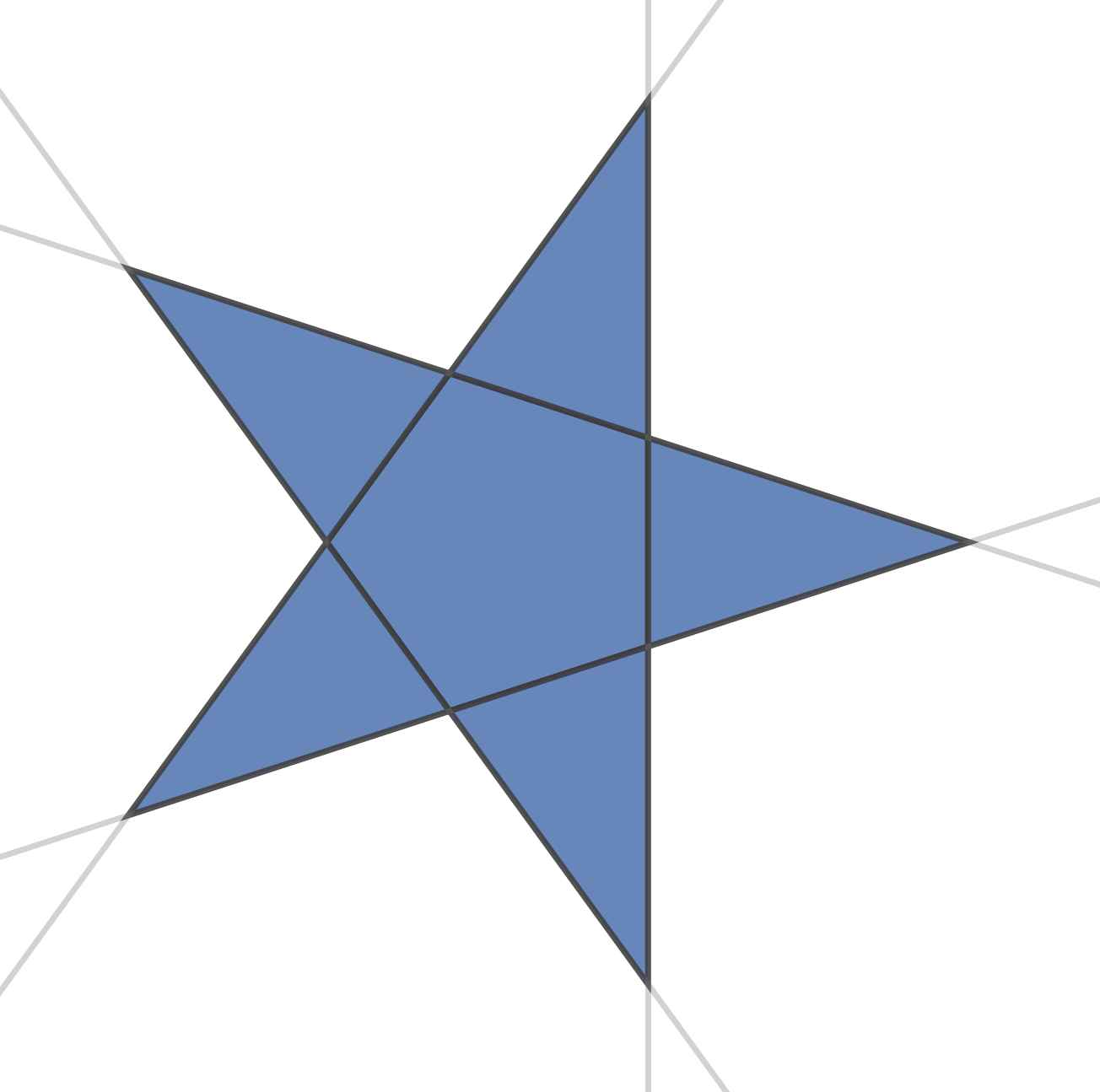}
        \caption{Star polytope.}
        \label{fig:star2}
    \end{subfigure}
    \hspace{0.03\textwidth}
    \begin{subfigure}[b]{0.38\textwidth}
        \centering
        \includegraphics[width=\textwidth]{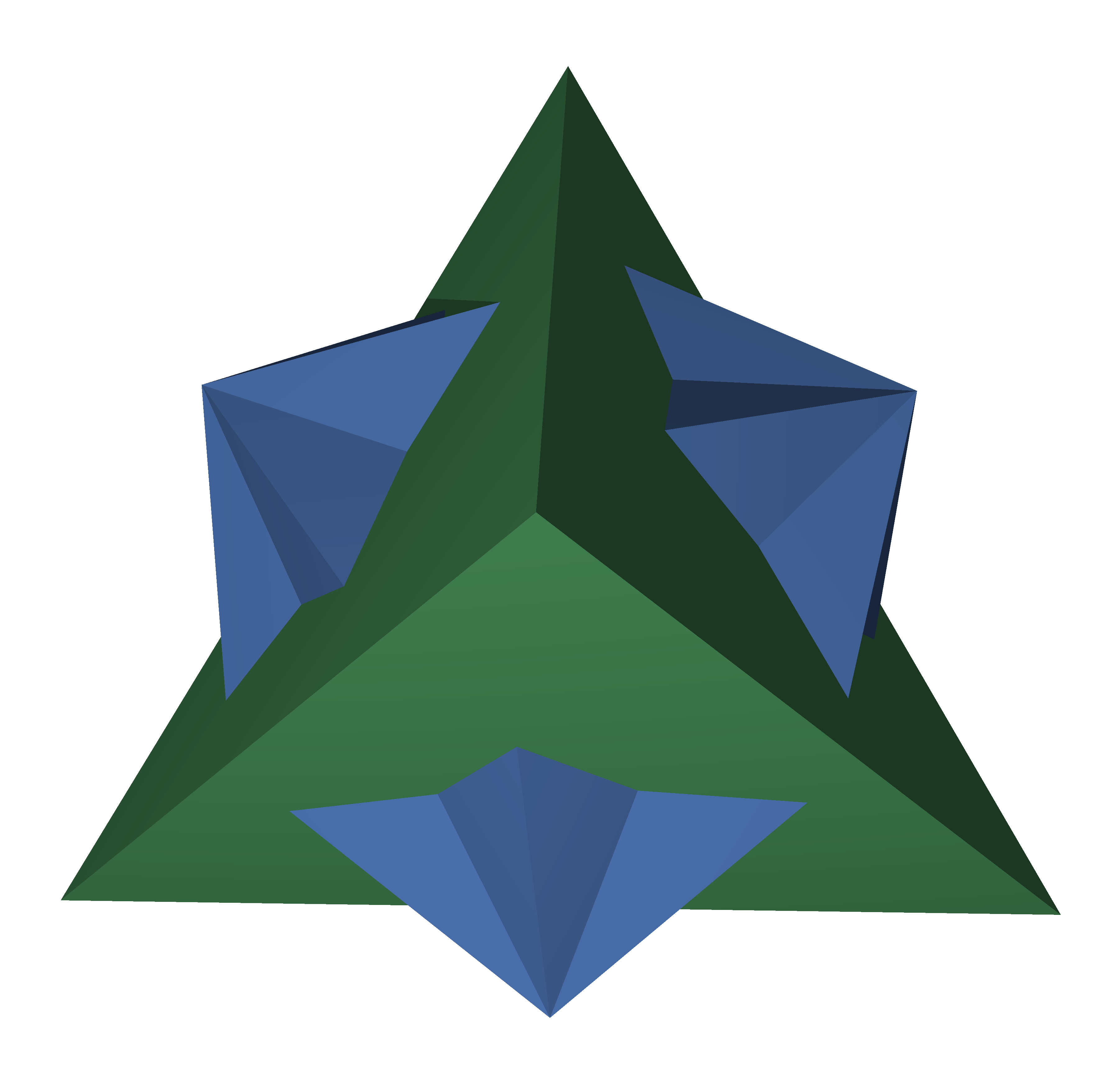}
        \caption{Tetrahedron.}
        \label{fig:tetra}
    \end{subfigure}
    \caption{Example polytopes.}
\end{figure}

Note that if there exists a set $S$ whose boundary has, possibly significantly,
fewer edges than predicted by the conjecture, then one of the following must hold:
\begin{itemize}
    \item Either it must have highly nontrivial stratum structure, with many disconnected strata, so as to
          deviate strongly from the conclusion of Proposition~\ref{prop:num_edges}.
          In fact, there exist polytopes even in $\R^3$ with surprising
          properties, for example having far fewer faces than their genus would suggest
          \cite{Ziegler2008HighGenus}.
          Below we show that at least in $\R^3$ such polytopes do not arise from
          set of chambers, since we prove that the conjecture holds
          asymptotically.
    \item On the other hand, even a strata-connected subset $S$ of
          chambers in $\R^d$ can have
          a stratification with $|b(A)| = 0$ for $|A| < d$. For example,
          consider the two-dimensional example from Figure~\ref{fig:star2},
          where $|b(\emptyset)| = 0$, or the three-dimensional example from Figure~\ref{fig:tetra},
          where $|b(\emptyset)| = 0$ and $|b(A)| = 0$ for many $A \in \supp(S)$
          (to be precise, this example does not come from an arrangement in general
          position, but one could add a sufficiently small random perturbation).
          Nevertheless, the examples do not negate the conjecture, since they
          have many 1-dimensional strata $T_S(A)$
          for which $|b(A)| = 2$, more than in a convex set attaining the conjectured
          lower bound.
\end{itemize}


\section{Low dimensions}

In the following results we establish the conjectured properties in low dimensions.

\begin{thm}\label{thm:R2_conjecture}
    Let $S$ be an arbitrary subset of at most half of the chambers of a hyperplane arrangement $\mathcal H$
    in general position in $\R^2$. Assume
    \[
        |S|=\sum_{i=0}^2 \binom{k}{i}
    \]
    for some $k\ge 1$. Then
    \[
        |\partial S|\ge \sum_{i=0}^1 \binom{k}{i}=k+1.
    \]
\end{thm}

\begin{proof}
    Since the target bound is concave (as a function of $|S|$) and $|\partial S|$ is additive over the
    connected components of $T_S$, it suffices to prove the theorem under the additional assumption
    that $T_S$ is connected.

    For $0\le i\le 2$, let
    \[
        N_i:=\sum_{\substack{A\in \supp(S)\\ |A|=i}} |\conn(T_S(A))|.
    \]
    Then $N_0=1$, and Proposition~\ref{prop:num_vertices} gives
    \[
        |S|\le N_0+N_1+N_2.
    \]
    Since every $2$-stratum is a point, we have
    \[
        N_2=|\supp(S)^{(2)}|.
    \]
    Moreover, $\supp(S)$ is a down-set and
    \[
        |\supp(S)|\le |S|.
    \]
    Therefore Theorem~\ref{thm:kruskal_katona}(i) with $m=2$ yields
    \[
        N_2 \le \binom{k}{2}.
    \]
    Hence
    \begin{equation}\label{eq:R2_N0N1_lb}
        N_0+N_1 \ge |S|-N_2 \ge |S| -\binom{k}{2}=k+1.
    \end{equation}

    We say that a line $H_f$ is \emph{fully supported} in $S$ if
    \[
        H_f \cap T_S \neq \emptyset \quad \text{and} \quad H_f\subseteq \overline{T_S}.
    \]

    Assume first that no line is fully supported in $S$. For a line $H_f$, let $\mathcal I_f$ be the set of
    connected components of $\Sigma_S\cap H_f$ having nonempty relative interior in $H_f$, and let
    $X_f$ be the set of intersection points of the arrangement that lie in the relative interior
    of one of these components, see Figure~\ref{fig:If_Xf}. Hence
    \[
        |\partial S|=\sum_f \bigl(|\mathcal I_f|+|X_f|\bigr).
    \]

    \begin{figure}[ht]
        \centering
        \includegraphics[width=0.55\textwidth]{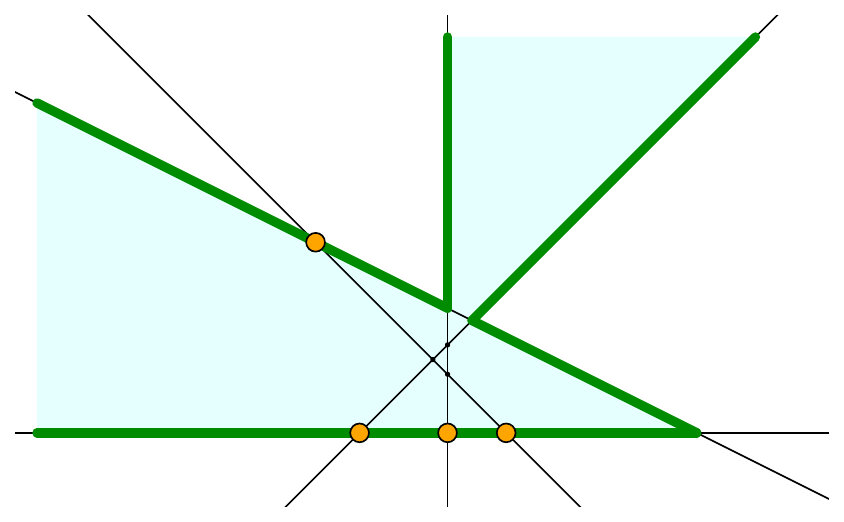}
        \caption{An arrangement in general position with $|S|=8$ (shaded) chambers.
            The boundary $\Sigma_S=\partial T_S$ decomposes along each line $H_f$
            into intervals $\mathcal I_f$ (green) and intersection points $X_f$ (orange).
            In this example $\sum_f|\mathcal I_f|=5$, $\sum_f|X_f|=4$, giving $|\partial S|=9$.
            Three components of the $1$-strata are covered by points, and the remaining component
            is covered by an interval. The four unbounded intervals remain for the 0-stratum.
        }
        \label{fig:If_Xf}
    \end{figure}

    We next cover the connected components of the $1$-strata. We say that a component
    $K\in \conn(T_S(\{f\}))$ is \emph{covered by points} if one of its endpoints belongs to
    $X_f$ for some line $H_f$. Every point of every $X_f$ can be an endpoint of at most one such
    component, so the total number of components covered by points is at most
    \[
        \sum_{f\in[n]} |X_f|.
    \]

    Fix a line $H_f$ and let $U_f\subseteq \conn(T_S(\{f\}))$ be the set of components not covered
    by points. Order them along $H_f$ as
    \[
        K_1,\dots,K_r .
    \]
    Between two consecutive components $K_j$ and $K_{j+1}$ there must be a boundary interval on
    $H_f$; otherwise the gap between them contains no boundary interval, and then the two endpoints
    of that gap belong to some sets $X_g$, so both $K_j$ and $K_{j+1}$ were already covered.
    Thus each such gap contributes an element of $\mathcal I_f$.

    Looking only at the components in $U_f$ and the intervals of $\mathcal I_f$ adjacent to them,
    we obtain a collection of alternating chains along $H_f$. Since $H_f$ is not fully supported, none of
    these chains can begin and end with components of $U_f$. Therefore the number of intervals in
    $\mathcal I_f$ adjacent to $U_f$ is at least $|U_f|$, and hence
    \[
        |\mathcal I_f|\ge |U_f|.
    \]

    Summing over all $f$, we see that the sets $X_f$ and $\mathcal I_f$ together cover all connected
    components of all $1$-strata. If $T_S$ is bounded, then each $\Sigma_S\cap H_f$ contributes
    at least one additional counted object.
    If $T_S$ is unbounded, then $\Sigma_S$ is unbounded. This implies that
    there exists a line $H_f$ such that $\Sigma_S\cap H_f$ has
    an unbounded component. Since $H_f$ is not fully supported,
    $\Sigma_S\cap H_f$ contributes at least
    one additional counted object.

    In both cases, an additional counted object can be assigned to the unique component of
    the $0$-stratum $T_S(\emptyset)=T_S$. Therefore
    \[
        |\partial S|\ge N_0+N_1.
    \]
    Together with \eqref{eq:R2_N0N1_lb} and $N_0=1$, this yields
    \[
        |\partial S|\ge k+1.
    \]

    It remains to consider the case when some line $H_e$ of $S$ is fully supported. Then it is easy to see
    that the complement $S^c$ of $S$ cannot have a line $H_f$ with $H_f\subseteq \overline{T_{S^c}}$, besides $H_e$.
    This implies that $S^c$ has no fully supported line, and the above arguments apply to $S^c$.
    Since $|\partial S| = |\partial S^c|$ and  $|S^c| \geq   |S|$,
    the assertion of the theorem follows.
\end{proof}

We remark that the assumption that $\mathcal H$ is in general position is essential in
Theorem~\ref{thm:R2_conjecture}. Choose points $p_1,\dots,p_k$ on a circle in cyclic order so that
no three diagonals of the convex polygon $P:=\conv\{p_1,\dots,p_k\}$ meet in the interior of $P$,
and let $\mathcal H$ consist of all lines through pairs $p_i,p_j$.
Then
\[
    |\mathcal H|=\binom{k}{2}=\Theta(k^2).
\]
Let $S$ be the set of chambers of $\mathcal H$ contained in $P$.
These chambers are exactly the regions into which the diagonals of the convex $k$--gon $P$
divide $P$, so
\[
    |S|=\binom{k}{4}+\binom{k}{2}-k+1=\Theta(k^4).
\]
On the other hand, each side of $P$ lies on a line of $\mathcal H$ and contains no intersection
point of any other line of the arrangement in its relative interior. Hence each side of $P$
is a facet shared by exactly one chamber of $S$ and one chamber outside $P$, and therefore
\[
    |\partial S|=k.
\]
Thus without the general-position assumption one can have
\[
    |\partial S|=\Theta(|S|^{1/4})
\]
already in dimension $2$.

\begin{thm}\label{thm:R3_asymp}
    Let $S$ be a subset of chambers of a hyperplane arrangement $\mathcal H$ in general position in $\R^3$.
    Assume
    \[
        |S|=\sum_{i=0}^3 \binom{k}{i}
    \]
    for some $k\ge 2$, and that $|S|\le D\cdot |\mathcal R(\mathcal H)|$ for a fixed constant $D\in(0,1)$, where
    $\mathcal R(\mathcal H)$ denotes the set of all chambers.
    Then there exists a constant $C=C(D)>0$ such that
    \[
        |\partial S|\;\ge\; C \sum_{\substack{A\in\supp(S)\\ 0\le |A|\le 2}}|\conn(T_S(A))|.
    \]
    In particular,
    \[
        |\partial S|\;\ge\; C\sum_{i=0}^2\binom{k}{i}.
    \]
    Hence $|\partial S|=\Omega( |S|^{2/3})$.
\end{thm}

\begin{proof}
    Let $T_S$ and $\Sigma_S=\partial T_S$ be as in Definitions~\ref{def:topological}
    and~\ref{def:boundary}.

    We first place a probabilistic charge on the boundary vertices.
    Every edge of the chamber graph crossing $\partial S$ corresponds to a $2$--dimensional face of $\Sigma_S$
    contained in some hyperplane of $\mathcal H$. Let $\mathcal F$ denote the set of these faces
    (so $|\mathcal F|=|\partial S|$). Let
    \[
        P:=\{H_{e_1}\cap H_{e_2}\cap H_{e_3}:e_1,e_2,e_3\in[n]\}\cap \Sigma_S
    \]
    be the set of vertices of the polyhedral complex on $\Sigma_S$.
    In general position, each $p\in P$ lies on exactly three hyperplanes of $\mathcal H$.

    Fix $F\in\mathcal F$. Choose a random direction $v\in S^2$ uniformly. If the linear functional
    $x\mapsto x\cdot v$ attains its minimum on $F$ at a vertex, choose one such minimizing vertex
    and denote it by $\low_v(F)$. Otherwise $F$ is unbounded, and we choose $\low_v(F)$ uniformly
    among the two vertices of $P \cap \partial F$ incident with the two unbounded edges of $F$.
    For $p\in P$ define
    \[
        \tau(p):=\sum_{F\in\mathcal F:\,p\in \partial F}\Pr[\low_v(F)=p].
    \]
    By construction,
    \begin{equation}
        \sum_{p\in P}\tau(p)=|\mathcal F|=|\partial S|.
    \end{equation}

    \begin{claim}\label{claim:angle_formula_R3}
        Let $F\in\mathcal F$ lie in an affine plane $H$, and let $p$ be a vertex of $F$.
        Write $\alpha_F(p):=\pi-\angle_F(p)$ for the exterior angle of $F$ at $p$ in $H$.
        Then
        \[
            \Pr[\low_v(F)=p]\ge \frac{\alpha_F(p)}{2\pi}.
        \]
        If $F$ is bounded, equality holds.
    \end{claim}

    \begin{proof}
        Project $v$ orthogonally to the direction circle in $H$. By rotational symmetry, the
        projected direction is uniform on $S^1\subset H$. The set of projected directions for which
        $p$ is an actual minimizer of $x\mapsto x\cdot v$ on $F$ is precisely the normal cone of $F$
        at $p$, whose angular measure is $\alpha_F(p)$. Thus these directions already contribute
        $\alpha_F(p)/(2\pi)$ to $\Pr[\low_v(F)=p]$. If $F$ is bounded, there are no extra directions,
        so equality holds.
    \end{proof}

    \begin{claim}\label{claim:unbounded_face_lb_R3}
        Let $F\in\mathcal F$ be unbounded. Then each of the two vertices incident with the two unbounded edges
        of $F$ satisfies
        \[
            \Pr[\low_v(F)=p]\ge \frac14.
        \]
    \end{claim}

    \begin{proof}
        If $F$ has no $v$--minimum, our rule selects one of the two distinguished vertices uniformly.
        The set of directions with no minimum has spherical measure at least $\frac12$, so each distinguished
        vertex gets probability at least $\frac12\cdot\frac12=\frac14$ from this case alone.
    \end{proof}

    \begin{claim}\label{claim:neighbor_pair_R3}
        Let $p\in P$, and let $F_1,F_2\in\mathcal F$ be incident with $p$ and contained in the same plane
        $H$. Assume that $F_1$ and $F_2$ are neighboring around $p$ inside $H$, i.e.\ their sectors in $H$
        share a boundary ray. Then
        \[
            \Pr[\low_v(F_1)=p]+\Pr[\low_v(F_2)=p]\ge \frac12.
        \]
    \end{claim}

    \begin{proof}
        Since the two sectors are neighboring, their interior angles satisfy
        $\angle_{F_1}(p)+\angle_{F_2}(p)=\pi$, hence
        $\alpha_{F_1}(p)+\alpha_{F_2}(p)=\pi$.
        The claim follows from Claim~\ref{claim:angle_formula_R3}.
    \end{proof}

    We now discharge this charge onto the lower-dimensional strata.
    Let
    \[
        \mathcal S_1:=\bigcup_{\substack{A\in\supp(S)\\ |A|=2}}\conn(T_S(A)),
        \qquad
        \mathcal S_2:=\bigcup_{\substack{A\in\supp(S)\\ |A|=1}}\conn(T_S(A)).
    \]
    For $p\in P$, let $\Inc(p)$ be the set of strata-components in $\mathcal S_1\cup\mathcal S_2$
    whose closure contains $p$.
    Since $p$ lies on exactly three hyperplanes, at most three lines and at most three planes
    from the stratification pass through $p$, and therefore
    \begin{equation}\label{eq:bounded_inc_degree_R3}
        |\Inc(p)|\le 6 \qquad\text{for all }p\in P.
    \end{equation}

    Split the charge $\tau(p)$ equally among the elements of $\Inc(p)$:
    each $p\in P$ sends $\tau(p)/|\Inc(p)|$ to every $s\in\Inc(p)$.
    (If $\Inc(p)=\emptyset$, keep the charge at $p$; this only strengthens the final inequality.)
    For $s\in\mathcal S_1\cup\mathcal S_2$ set
    \[
        \ch(s):=\sum_{p\in P\cap\overline{s}}\frac{\tau(p)}{|\Inc(p)|}.
    \]
    Then
    \begin{equation}\label{eq:boundary_ge_sum_ch_R3}
        |\partial S|
        =
        \sum_{p\in P}\tau(p)
        \ge
        \sum_{s\in\mathcal S_1\cup\mathcal S_2}\ch(s).
    \end{equation}

    Call a component of $\mathcal S_1$ \emph{full} if it is an entire line
    $H_e\cap H_f$, and call a component of $\mathcal S_2$ \emph{full} if it
    is an entire plane $H_e$. All other components are called \emph{non-full}.

    We next obtain local lower bounds for this charge. We first treat endpoints of
    $1$--dimensional strata.

    \begin{claim}\label{claim:tau_endpoint_lb_R3}
        Let $p\in P$ and assume that $p$ is an endpoint of some $s\in\conn(T_S(A))$ with $|A|=2$.
        Then
        \[
            \tau(p)\ge \frac12.
        \]
    \end{claim}

    \begin{proof}
        Write $p=H_a\cap H_b\cap H_c$, and assume that the endpoint lies on the line $H_a\cap H_b$.
        In a sufficiently small ball around $p$, the eight incident chambers form a $3$--cube $Q_p$.
        The endpoint hypothesis means that one of the two open half-lines of $H_a\cap H_b$ emanating from $p$
        is contained in $T_S(\{a,b\})$. Equivalently, the four chambers of one $H_c$--layer of $Q_p$ form a square
        contained in $S$. Since $p\in\Sigma_S$, not all eight chambers belong to $S$.

        Up to cube symmetries preserving this distinguished square, exactly five local patterns can occur;
        see Figure~\ref{fig:endpoint_cases}:
        \begin{enumerate}[(i)]
            \item only the square lies in $S$;
            \item the square together with one additional chamber lies in $S$;
            \item the square together with two additional chambers lies in $S$, with the two missing chambers adjacent;
            \item the square together with two additional chambers lies in $S$, with the two missing chambers opposite;
            \item all chambers except one lie in $S$.
        \end{enumerate}

\begin{figure}[ht]
\centering
\begin{subfigure}{0.19\textwidth}
\centering
\begin{tikzpicture}[scale=0.65]
  \draw[thick,red!70!black,opacity=0.3] (0.40,0.60) -- (0.72,1.08);
  \draw[gray!50,opacity=0.3] (0.80,2.60) -- (0.80,1.34);
  \draw[gray!50,opacity=0.3] (2.40,1.20) -- (0.94,1.20);
  \node[circle,fill=white,draw=black,inner sep=2pt,minimum size=5pt,opacity=0.35] (001) at (0.8,1.2) {};
  \fill[red!50,fill opacity=0.18] (1.80,1.20) -- (1.80,3.20) -- (1.00,2.00) -- (1.00,0.00) -- cycle;
  \draw[red!70!black,thin,opacity=0.5] (1.80,1.20) -- (1.80,3.20) -- (1.00,2.00) -- (1.00,0.00) -- cycle;
  \fill[green!40,fill opacity=0.18] (0.00,1.00) -- (2.00,1.00) -- (2.80,2.20) -- (0.80,2.20) -- cycle;
  \draw[green!60!black,thin,opacity=0.5] (0.00,1.00) -- (2.00,1.00) -- (2.80,2.20) -- (0.80,2.20) -- cycle;
  \fill[orange!50,fill opacity=0.18] (0.40,0.60) -- (2.40,0.60) -- (2.40,2.60) -- (0.40,2.60) -- cycle;
  \draw[orange!70!black,thin,opacity=0.5] (0.40,0.60) -- (2.40,0.60) -- (2.40,2.60) -- (0.40,2.60) -- cycle;
  \draw[gray!50] (0.00,0.00) -- (2.00,0.00);
  \draw[gray!50] (0.00,0.00) -- (0.00,2.00);
  \draw[thick,red!70!black] (0.00,0.00) -- (0.40,0.60);
  \draw[gray!50] (2.00,2.00) -- (2.00,0.00);
  \draw[gray!50] (2.00,2.00) -- (0.00,2.00);
  \draw[thick,red!70!black] (2.00,2.00) -- (2.80,3.20);
  \draw[thick,red!70!black] (0.80,3.20) -- (0.00,2.00);
  \draw[gray!50] (0.80,3.20) -- (0.80,2.60);
  \draw[gray!50] (0.80,3.20) -- (2.80,3.20);
  \draw[thick,red!70!black] (2.80,1.20) -- (2.00,0.00);
  \draw[gray!50] (2.80,1.20) -- (2.40,1.20);
  \draw[gray!50] (2.80,1.20) -- (2.80,3.20);
  \node[circle,fill=blue!60,draw=black,inner sep=2pt,minimum size=5pt] (000) at (0.0,0.0) {};
  \node[circle,fill=blue!60,draw=black,inner sep=2pt,minimum size=5pt] (100) at (2.0,0.0) {};
  \node[circle,fill=blue!60,draw=black,inner sep=2pt,minimum size=5pt] (010) at (0.0,2.0) {};
  \node[circle,fill=blue!60,draw=black,inner sep=2pt,minimum size=5pt] (110) at (2.0,2.0) {};
  \node[circle,fill=white,draw=black,inner sep=2pt,minimum size=5pt] (101) at (2.8,1.2) {};
  \node[circle,fill=white,draw=black,inner sep=2pt,minimum size=5pt] (011) at (0.8,3.2) {};
  \node[circle,fill=white,draw=black,inner sep=2pt,minimum size=5pt] (111) at (2.8,3.2) {};
\end{tikzpicture}
\caption{(i)}
\end{subfigure}
\hfill
\begin{subfigure}{0.19\textwidth}
\centering
\begin{tikzpicture}[scale=0.65]
  \draw[gray!50,opacity=0.3] (0.40,0.60) -- (0.72,1.08);
  \draw[thick,red!70!black,opacity=0.3] (0.80,2.60) -- (0.80,1.34);
  \draw[thick,red!70!black,opacity=0.3] (2.40,1.20) -- (0.94,1.20);
  \node[circle,fill=blue!60,draw=black,inner sep=2pt,minimum size=5pt,opacity=0.35] (001) at (0.8,1.2) {};
  \fill[red!50,fill opacity=0.18] (1.80,1.20) -- (1.80,3.20) -- (1.00,2.00) -- (1.00,0.00) -- cycle;
  \draw[red!70!black,thin,opacity=0.5] (1.80,1.20) -- (1.80,3.20) -- (1.00,2.00) -- (1.00,0.00) -- cycle;
  \fill[green!40,fill opacity=0.18] (0.00,1.00) -- (2.00,1.00) -- (2.80,2.20) -- (0.80,2.20) -- cycle;
  \draw[green!60!black,thin,opacity=0.5] (0.00,1.00) -- (2.00,1.00) -- (2.80,2.20) -- (0.80,2.20) -- cycle;
  \fill[orange!50,fill opacity=0.18] (0.40,0.60) -- (2.40,0.60) -- (2.40,2.60) -- (0.40,2.60) -- cycle;
  \draw[orange!70!black,thin,opacity=0.5] (0.40,0.60) -- (2.40,0.60) -- (2.40,2.60) -- (0.40,2.60) -- cycle;
  \draw[gray!50] (0.00,0.00) -- (2.00,0.00);
  \draw[gray!50] (0.00,0.00) -- (0.00,2.00);
  \draw[gray!50] (0.00,0.00) -- (0.40,0.60);
  \draw[gray!50] (2.00,2.00) -- (2.00,0.00);
  \draw[gray!50] (2.00,2.00) -- (0.00,2.00);
  \draw[thick,red!70!black] (2.00,2.00) -- (2.80,3.20);
  \draw[thick,red!70!black] (0.80,3.20) -- (0.00,2.00);
  \draw[thick,red!70!black] (0.80,3.20) -- (0.80,2.60);
  \draw[gray!50] (0.80,3.20) -- (2.80,3.20);
  \draw[thick,red!70!black] (2.80,1.20) -- (2.00,0.00);
  \draw[thick,red!70!black] (2.80,1.20) -- (2.40,1.20);
  \draw[gray!50] (2.80,1.20) -- (2.80,3.20);
  \node[circle,fill=blue!60,draw=black,inner sep=2pt,minimum size=5pt] (000) at (0.0,0.0) {};
  \node[circle,fill=blue!60,draw=black,inner sep=2pt,minimum size=5pt] (100) at (2.0,0.0) {};
  \node[circle,fill=blue!60,draw=black,inner sep=2pt,minimum size=5pt] (010) at (0.0,2.0) {};
  \node[circle,fill=blue!60,draw=black,inner sep=2pt,minimum size=5pt] (110) at (2.0,2.0) {};
  \node[circle,fill=white,draw=black,inner sep=2pt,minimum size=5pt] (101) at (2.8,1.2) {};
  \node[circle,fill=white,draw=black,inner sep=2pt,minimum size=5pt] (011) at (0.8,3.2) {};
  \node[circle,fill=white,draw=black,inner sep=2pt,minimum size=5pt] (111) at (2.8,3.2) {};
\end{tikzpicture}
\caption{(ii)}
\end{subfigure}
\hfill
\begin{subfigure}{0.19\textwidth}
\centering
\begin{tikzpicture}[scale=0.65]
  \draw[gray!50,opacity=0.3] (0.40,0.60) -- (0.72,1.08);
  \draw[thick,red!70!black,opacity=0.3] (0.80,2.60) -- (0.80,1.34);
  \draw[gray!50,opacity=0.3] (2.40,1.20) -- (0.94,1.20);
  \node[circle,fill=blue!60,draw=black,inner sep=2pt,minimum size=5pt,opacity=0.35] (001) at (0.8,1.2) {};
  \fill[red!50,fill opacity=0.18] (1.80,1.20) -- (1.80,3.20) -- (1.00,2.00) -- (1.00,0.00) -- cycle;
  \draw[red!70!black,thin,opacity=0.5] (1.80,1.20) -- (1.80,3.20) -- (1.00,2.00) -- (1.00,0.00) -- cycle;
  \fill[green!40,fill opacity=0.18] (0.00,1.00) -- (2.00,1.00) -- (2.80,2.20) -- (0.80,2.20) -- cycle;
  \draw[green!60!black,thin,opacity=0.5] (0.00,1.00) -- (2.00,1.00) -- (2.80,2.20) -- (0.80,2.20) -- cycle;
  \fill[orange!50,fill opacity=0.18] (0.40,0.60) -- (2.40,0.60) -- (2.40,2.60) -- (0.40,2.60) -- cycle;
  \draw[orange!70!black,thin,opacity=0.5] (0.40,0.60) -- (2.40,0.60) -- (2.40,2.60) -- (0.40,2.60) -- cycle;
  \draw[gray!50] (0.00,0.00) -- (2.00,0.00);
  \draw[gray!50] (0.00,0.00) -- (0.00,2.00);
  \draw[gray!50] (0.00,0.00) -- (0.40,0.60);
  \draw[gray!50] (2.00,2.00) -- (2.00,0.00);
  \draw[gray!50] (2.00,2.00) -- (0.00,2.00);
  \draw[thick,red!70!black] (2.00,2.00) -- (2.80,3.20);
  \draw[thick,red!70!black] (0.80,3.20) -- (0.00,2.00);
  \draw[thick,red!70!black] (0.80,3.20) -- (0.80,2.60);
  \draw[gray!50] (0.80,3.20) -- (2.80,3.20);
  \draw[gray!50] (2.80,1.20) -- (2.00,0.00);
  \draw[gray!50] (2.80,1.20) -- (2.40,1.20);
  \draw[thick,red!70!black] (2.80,1.20) -- (2.80,3.20);
  \node[circle,fill=blue!60,draw=black,inner sep=2pt,minimum size=5pt] (000) at (0.0,0.0) {};
  \node[circle,fill=blue!60,draw=black,inner sep=2pt,minimum size=5pt] (100) at (2.0,0.0) {};
  \node[circle,fill=blue!60,draw=black,inner sep=2pt,minimum size=5pt] (010) at (0.0,2.0) {};
  \node[circle,fill=blue!60,draw=black,inner sep=2pt,minimum size=5pt] (110) at (2.0,2.0) {};
  \node[circle,fill=blue!60,draw=black,inner sep=2pt,minimum size=5pt] (101) at (2.8,1.2) {};
  \node[circle,fill=white,draw=black,inner sep=2pt,minimum size=5pt] (011) at (0.8,3.2) {};
  \node[circle,fill=white,draw=black,inner sep=2pt,minimum size=5pt] (111) at (2.8,3.2) {};
\end{tikzpicture}
\caption{(iii)}
\end{subfigure}
\hfill
\begin{subfigure}{0.19\textwidth}
\centering
\begin{tikzpicture}[scale=0.65]
  \draw[gray!50,opacity=0.3] (0.40,0.60) -- (0.72,1.08);
  \draw[thick,red!70!black,opacity=0.3] (0.80,2.60) -- (0.80,1.34);
  \draw[thick,red!70!black,opacity=0.3] (2.40,1.20) -- (0.94,1.20);
  \node[circle,fill=blue!60,draw=black,inner sep=2pt,minimum size=5pt,opacity=0.35] (001) at (0.8,1.2) {};
  \fill[red!50,fill opacity=0.18] (1.80,1.20) -- (1.80,3.20) -- (1.00,2.00) -- (1.00,0.00) -- cycle;
  \draw[red!70!black,thin,opacity=0.5] (1.80,1.20) -- (1.80,3.20) -- (1.00,2.00) -- (1.00,0.00) -- cycle;
  \fill[green!40,fill opacity=0.18] (0.00,1.00) -- (2.00,1.00) -- (2.80,2.20) -- (0.80,2.20) -- cycle;
  \draw[green!60!black,thin,opacity=0.5] (0.00,1.00) -- (2.00,1.00) -- (2.80,2.20) -- (0.80,2.20) -- cycle;
  \fill[orange!50,fill opacity=0.18] (0.40,0.60) -- (2.40,0.60) -- (2.40,2.60) -- (0.40,2.60) -- cycle;
  \draw[orange!70!black,thin,opacity=0.5] (0.40,0.60) -- (2.40,0.60) -- (2.40,2.60) -- (0.40,2.60) -- cycle;
  \draw[gray!50] (0.00,0.00) -- (2.00,0.00);
  \draw[gray!50] (0.00,0.00) -- (0.00,2.00);
  \draw[gray!50] (0.00,0.00) -- (0.40,0.60);
  \draw[gray!50] (2.00,2.00) -- (2.00,0.00);
  \draw[gray!50] (2.00,2.00) -- (0.00,2.00);
  \draw[gray!50] (2.00,2.00) -- (2.80,3.20);
  \draw[thick,red!70!black] (0.80,3.20) -- (0.00,2.00);
  \draw[thick,red!70!black] (0.80,3.20) -- (0.80,2.60);
  \draw[thick,red!70!black] (0.80,3.20) -- (2.80,3.20);
  \draw[thick,red!70!black] (2.80,1.20) -- (2.00,0.00);
  \draw[thick,red!70!black] (2.80,1.20) -- (2.40,1.20);
  \draw[thick,red!70!black] (2.80,1.20) -- (2.80,3.20);
  \node[circle,fill=blue!60,draw=black,inner sep=2pt,minimum size=5pt] (000) at (0.0,0.0) {};
  \node[circle,fill=blue!60,draw=black,inner sep=2pt,minimum size=5pt] (100) at (2.0,0.0) {};
  \node[circle,fill=blue!60,draw=black,inner sep=2pt,minimum size=5pt] (010) at (0.0,2.0) {};
  \node[circle,fill=blue!60,draw=black,inner sep=2pt,minimum size=5pt] (110) at (2.0,2.0) {};
  \node[circle,fill=white,draw=black,inner sep=2pt,minimum size=5pt] (101) at (2.8,1.2) {};
  \node[circle,fill=white,draw=black,inner sep=2pt,minimum size=5pt] (011) at (0.8,3.2) {};
  \node[circle,fill=blue!60,draw=black,inner sep=2pt,minimum size=5pt] (111) at (2.8,3.2) {};
\end{tikzpicture}
\caption{(iv)}
\end{subfigure}
\hfill
\begin{subfigure}{0.19\textwidth}
\centering
\begin{tikzpicture}[scale=0.65]
  \draw[gray!50,opacity=0.3] (0.40,0.60) -- (0.72,1.08);
  \draw[gray!50,opacity=0.3] (0.80,2.60) -- (0.80,1.34);
  \draw[gray!50,opacity=0.3] (2.40,1.20) -- (0.94,1.20);
  \node[circle,fill=blue!60,draw=black,inner sep=2pt,minimum size=5pt,opacity=0.35] (001) at (0.8,1.2) {};
  \fill[red!50,fill opacity=0.18] (1.80,1.20) -- (1.80,3.20) -- (1.00,2.00) -- (1.00,0.00) -- cycle;
  \draw[red!70!black,thin,opacity=0.5] (1.80,1.20) -- (1.80,3.20) -- (1.00,2.00) -- (1.00,0.00) -- cycle;
  \fill[green!40,fill opacity=0.18] (0.00,1.00) -- (2.00,1.00) -- (2.80,2.20) -- (0.80,2.20) -- cycle;
  \draw[green!60!black,thin,opacity=0.5] (0.00,1.00) -- (2.00,1.00) -- (2.80,2.20) -- (0.80,2.20) -- cycle;
  \fill[orange!50,fill opacity=0.18] (0.40,0.60) -- (2.40,0.60) -- (2.40,2.60) -- (0.40,2.60) -- cycle;
  \draw[orange!70!black,thin,opacity=0.5] (0.40,0.60) -- (2.40,0.60) -- (2.40,2.60) -- (0.40,2.60) -- cycle;
  \draw[gray!50] (0.00,0.00) -- (2.00,0.00);
  \draw[gray!50] (0.00,0.00) -- (0.00,2.00);
  \draw[gray!50] (0.00,0.00) -- (0.40,0.60);
  \draw[gray!50] (2.00,2.00) -- (2.00,0.00);
  \draw[gray!50] (2.00,2.00) -- (0.00,2.00);
  \draw[thick,red!70!black] (2.00,2.00) -- (2.80,3.20);
  \draw[gray!50] (0.80,3.20) -- (0.00,2.00);
  \draw[gray!50] (0.80,3.20) -- (0.80,2.60);
  \draw[thick,red!70!black] (0.80,3.20) -- (2.80,3.20);
  \draw[gray!50] (2.80,1.20) -- (2.00,0.00);
  \draw[gray!50] (2.80,1.20) -- (2.40,1.20);
  \draw[thick,red!70!black] (2.80,1.20) -- (2.80,3.20);
  \node[circle,fill=blue!60,draw=black,inner sep=2pt,minimum size=5pt] (000) at (0.0,0.0) {};
  \node[circle,fill=blue!60,draw=black,inner sep=2pt,minimum size=5pt] (100) at (2.0,0.0) {};
  \node[circle,fill=blue!60,draw=black,inner sep=2pt,minimum size=5pt] (010) at (0.0,2.0) {};
  \node[circle,fill=blue!60,draw=black,inner sep=2pt,minimum size=5pt] (110) at (2.0,2.0) {};
  \node[circle,fill=blue!60,draw=black,inner sep=2pt,minimum size=5pt] (101) at (2.8,1.2) {};
  \node[circle,fill=blue!60,draw=black,inner sep=2pt,minimum size=5pt] (011) at (0.8,3.2) {};
  \node[circle,fill=white,draw=black,inner sep=2pt,minimum size=5pt] (111) at (2.8,3.2) {};
\end{tikzpicture}
\caption{(v)}
\end{subfigure}
\caption{Local configurations at an endpoint of a $1$--stratum component (Claim~\ref{claim:tau_endpoint_lb_R3}). Filled vertices represent chambers in~$S$; hollow vertices are chambers not in~$S$. Red edges cross the boundary~$\partial S$. The bottom face is the distinguished square (four chambers in one $H_c$--layer all belonging to~$S$).}
\label{fig:endpoint_cases}
\end{figure}
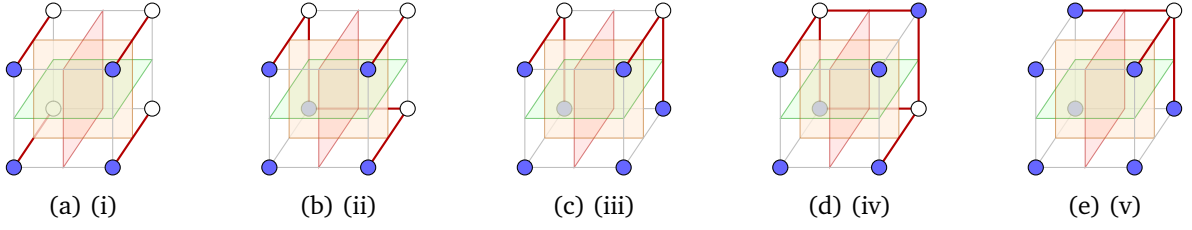

        In case (i), the four boundary faces at $p$ lie in the same plane, so their exterior angles sum to $2\pi$;
        Claim~\ref{claim:angle_formula_R3} gives $\tau(p)\ge 1$.

        In case (ii), the three missing chambers all lie in the $H_c$--layer opposite to the distinguished square.
        Hence, in the plane $H_c$, there are at least two neighboring boundary faces at $p$, so
        Claim~\ref{claim:neighbor_pair_R3} gives $\tau(p)\ge \frac12$.

        In case (iii), one of the planes $H_a,H_b,H_c$ contains two neighboring boundary faces,
        so Claim~\ref{claim:neighbor_pair_R3} gives $\tau(p)\ge \frac12$.

        In case (iv), the two opposite missing chambers define two disjoint local missing cones,
        each contributing at least $\frac12$ by the previous paragraph. Hence $\tau(p)\ge 1$.

        In case (v), there is a unique missing chamber. The three boundary faces separating this chamber from the
        three adjacent chambers of $S$ form a trihedral configuration. The sum of their planar angles is at most $2\pi$,
        so the sum of their exterior angles is at least $\pi$. Claim~\ref{claim:angle_formula_R3} therefore yields
        $\tau(p)\ge \frac12$.
    \end{proof}

    \begin{claim}\label{clm:S1_charge_R3}
        Every non-full component $s\in\mathcal S_1$ satisfies $\ch(s)\ge \frac1{12}$.
    \end{claim}

    \begin{proof}
        A non-full connected open subset of a line is an interval or a ray, hence has at least one endpoint
        $p\in P$. By Claim~\ref{claim:tau_endpoint_lb_R3}, $\tau(p)\ge \frac12$, and thus
        \[
            \ch(s)\ge \frac{\tau(p)}{|\Inc(p)|}\ge \frac{1/2}{6}=\frac1{12}
        \]
        by \eqref{eq:bounded_inc_degree_R3}.
    \end{proof}

    Now let $s\in\mathcal S_2$ be a connected component of $T_S(\{e\})\subset H_e$ for some $e$.
    Assume that $\overline{s}^{\,H_e}\neq H_e$, and write
    \[
        \partial s:=\partial_{H_e}s,
        \qquad
        V(s):=P\cap\partial s.
    \]

    \begin{claim}\label{clm:S2_charge_R3}
        Every non-full component $s\in\mathcal S_2$ satisfies $\ch(s)\ge \frac1{12}$.
    \end{claim}

    \begin{proof}
        We show that
        \begin{equation}\label{eq:sumtau_Vs_lb_R3}
            \sum_{p\in V(s)}\tau(p)\ge \frac12.
        \end{equation}
        Then \eqref{eq:bounded_inc_degree_R3} implies
        \[
            \ch(s)\ge \frac16\sum_{p\in V(s)}\tau(p)\ge \frac1{12}.
        \]

        If $s$ is unbounded, then its polygonal boundary in $H_e$ has two ends. At each end there is a vertex
        incident with an unbounded face of $\mathcal F$, so Claim~\ref{claim:unbounded_face_lb_R3} gives two vertices
        with charge at least $\frac14$. Hence \eqref{eq:sumtau_Vs_lb_R3} holds.

        Now assume that $s$ is bounded. If there exist $f$ and
        $s'\in\conn(T_S(\{e,f\}))$ with
        \[
            \overline{s'}^{\,H_e\cap H_f}\cap \partial s\neq \emptyset,
        \]
        then any point $p$ in this intersection is an endpoint of the $1$--stratum component $s'$,
        so Claim~\ref{claim:tau_endpoint_lb_R3} yields $\tau(p)\ge \frac12$ and we are done.

        Thus we may assume that
        \begin{equation}\label{eq:no_s1_accumulation_R3}
            \overline{s'}^{\,H_e\cap H_f}\cap \partial s=\emptyset
            \qquad\text{for every $f$ and every }s'\in\conn(T_S(\{e,f\})).
        \end{equation}

        We claim that under \eqref{eq:no_s1_accumulation_R3}, the set $s$ is a bounded convex polygon in $H_e$.
        Indeed, fix $p\in V(s)$ and write $p=H_e\cap H_a\cap H_b$. In the plane $H_e$, the two lines
        $H_e\cap H_a$ and $H_e\cap H_b$ cut a neighborhood of $p$ into four sectors. Since $s$ is connected and
        $p$ is a boundary vertex of $s$, the local intersection $s\cap H_e$ can consist only of one sector, two
        adjacent sectors, or three consecutive sectors. If two adjacent sectors are occupied, then their common
        boundary half-line in $H_e\cap H_a$ or $H_e\cap H_b$ is contained in a $1$--stratum component of $T_S$
        whose closure meets $p$, contradicting \eqref{eq:no_s1_accumulation_R3}. The same contradiction arises if
        three consecutive sectors are occupied, since then one of the two half-lines of $H_e\cap H_a$ or
        $H_e\cap H_b$ emanating from $p$ lies in a $1$--stratum component of $T_S$. Thus exactly one sector is
        occupied at each boundary vertex.
        Consequently every interior angle of $\partial s$ is $<\pi$.
        Each boundary component of $s$ is therefore a simple polygon with all interior angles $<\pi$, hence convex.
        Since $s$ is connected and bounded, there can be no inner boundary component: viewed from $s$,
        a vertex on an inner component would have interior angle $>\pi$. Thus $\partial s$ is a single convex polygon.

        For $p\in V(s)$, let $r_s(p)\in(0,\pi)$ denote the interior angle of $s$ at $p$.
        Let $Q_p$ be the local $3$--cube around $p$, determined by the hyperplanes $H_e,H_a,H_b$, and set
        $X:=S\cap V(Q_p)$.

        \begin{claim}\label{clm:local_S2_dichotomy_R3}
            For every $p\in V(s)$ under \eqref{eq:no_s1_accumulation_R3}, one of the following holds:
            \begin{enumerate}[(a)]
                \item $\tau(p)\ge \frac12$;
                \item up to the symmetries of $Q_p$ preserving the two $H_e$--layers,
                      \[
                          X=\{000,001,010,100\},
                      \]
                      and in this exceptional configuration
                      \[
                          \tau(p)\ge \frac{r_s(p)}{\pi}.
                      \]
            \end{enumerate}
        \end{claim}

        \begin{proof}
            Since exactly one sector of $H_e$ is occupied by $s$ near $p$, after relabeling we may assume
            that this sector is the one corresponding to the pair of chambers $000$ and $100$, so these two
            vertices belong to $X$.
            The assumption \eqref{eq:no_s1_accumulation_R3} implies that no full $y$--slice or $z$--slice of $Q_p$
            lies in $X$; equivalently, $p$ is not an endpoint of a $1$--stratum component. Hence $2\le |X|\le 5$.

            If $|X|=2$, then $X=\{000,100\}$ and both planes $H_a$ and $H_b$ contain a neighboring pair of
            boundary faces at $p$. Claim~\ref{claim:neighbor_pair_R3} gives $\tau(p)\ge 1$.

            If $|X|=3$, then up to swapping $H_a$ and $H_b$ there are only two possibilities:
            \[
                X=\{000,001,100\}
                \qquad\text{or}\qquad
                X=\{000,011,100\}.
            \]
            In each case one of the planes $H_a$ or $H_b$ contains a neighboring pair of boundary faces,
            so Claim~\ref{claim:neighbor_pair_R3} gives $\tau(p)\ge \frac12$.

            If $|X|=5$, replace $X$ by its complement in $Q_p$. The set of boundary faces incident with $p$
            is unchanged, so this case is the complement of the previous one and again yields $\tau(p)\ge \frac12$.

            It remains to consider $|X|=4$. Up to the same symmetries, there are four admissible patterns;
            see Figure~\ref{fig:S2_dichotomy_cases}:
            \[
                \{000,001,010,100\},\quad
                \{000,001,100,110\},\quad
                \{000,001,010,101\},\quad
                \{000,001,100,111\}.
            \]

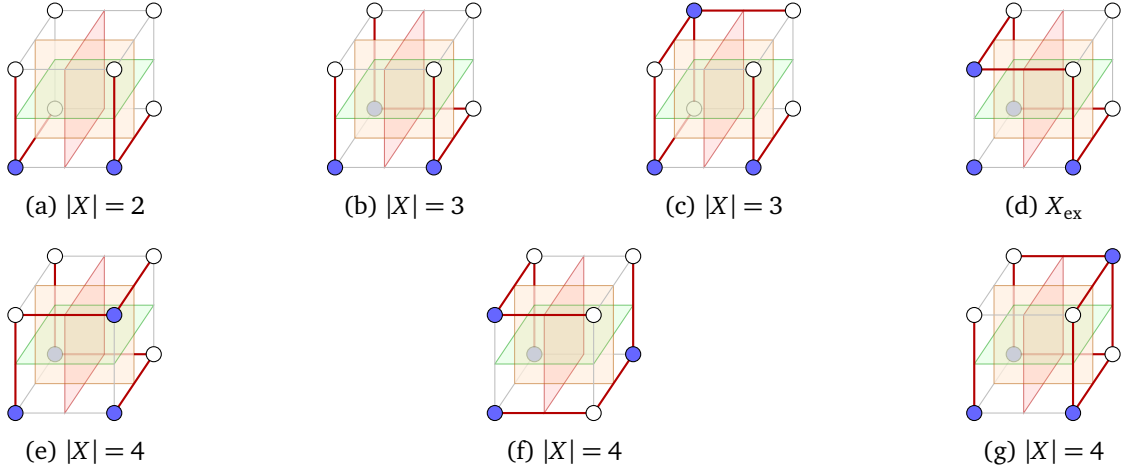
\begin{figure}[ht]
\centering
\begin{subfigure}{0.24\textwidth}
\centering
\begin{tikzpicture}[scale=0.65]
  \draw[thick,red!70!black,opacity=0.3] (0.40,0.60) -- (0.72,1.08);
  \draw[gray!50,opacity=0.3] (0.80,2.60) -- (0.80,1.34);
  \draw[gray!50,opacity=0.3] (2.40,1.20) -- (0.94,1.20);
  \node[circle,fill=white,draw=black,inner sep=2pt,minimum size=5pt,opacity=0.35] (001) at (0.8,1.2) {};
  \fill[red!50,fill opacity=0.18] (1.80,1.20) -- (1.80,3.20) -- (1.00,2.00) -- (1.00,0.00) -- cycle;
  \draw[red!70!black,thin,opacity=0.5] (1.80,1.20) -- (1.80,3.20) -- (1.00,2.00) -- (1.00,0.00) -- cycle;
  \fill[green!40,fill opacity=0.18] (0.00,1.00) -- (2.00,1.00) -- (2.80,2.20) -- (0.80,2.20) -- cycle;
  \draw[green!60!black,thin,opacity=0.5] (0.00,1.00) -- (2.00,1.00) -- (2.80,2.20) -- (0.80,2.20) -- cycle;
  \fill[orange!50,fill opacity=0.18] (0.40,0.60) -- (2.40,0.60) -- (2.40,2.60) -- (0.40,2.60) -- cycle;
  \draw[orange!70!black,thin,opacity=0.5] (0.40,0.60) -- (2.40,0.60) -- (2.40,2.60) -- (0.40,2.60) -- cycle;
  \draw[gray!50] (0.00,0.00) -- (2.00,0.00);
  \draw[thick,red!70!black] (0.00,0.00) -- (0.00,2.00);
  \draw[thick,red!70!black] (0.00,0.00) -- (0.40,0.60);
  \draw[thick,red!70!black] (2.00,2.00) -- (2.00,0.00);
  \draw[gray!50] (2.00,2.00) -- (0.00,2.00);
  \draw[gray!50] (2.00,2.00) -- (2.80,3.20);
  \draw[gray!50] (0.80,3.20) -- (0.00,2.00);
  \draw[gray!50] (0.80,3.20) -- (0.80,2.60);
  \draw[gray!50] (0.80,3.20) -- (2.80,3.20);
  \draw[thick,red!70!black] (2.80,1.20) -- (2.00,0.00);
  \draw[gray!50] (2.80,1.20) -- (2.40,1.20);
  \draw[gray!50] (2.80,1.20) -- (2.80,3.20);
  \node[circle,fill=blue!60,draw=black,inner sep=2pt,minimum size=5pt] (000) at (0.0,0.0) {};
  \node[circle,fill=blue!60,draw=black,inner sep=2pt,minimum size=5pt] (100) at (2.0,0.0) {};
  \node[circle,fill=white,draw=black,inner sep=2pt,minimum size=5pt] (010) at (0.0,2.0) {};
  \node[circle,fill=white,draw=black,inner sep=2pt,minimum size=5pt] (110) at (2.0,2.0) {};
  \node[circle,fill=white,draw=black,inner sep=2pt,minimum size=5pt] (101) at (2.8,1.2) {};
  \node[circle,fill=white,draw=black,inner sep=2pt,minimum size=5pt] (011) at (0.8,3.2) {};
  \node[circle,fill=white,draw=black,inner sep=2pt,minimum size=5pt] (111) at (2.8,3.2) {};
\end{tikzpicture}
\caption{$|X|=2$}
\end{subfigure}
\hfill
\begin{subfigure}{0.24\textwidth}
\centering
\begin{tikzpicture}[scale=0.65]
  \draw[gray!50,opacity=0.3] (0.40,0.60) -- (0.72,1.08);
  \draw[thick,red!70!black,opacity=0.3] (0.80,2.60) -- (0.80,1.34);
  \draw[thick,red!70!black,opacity=0.3] (2.40,1.20) -- (0.94,1.20);
  \node[circle,fill=blue!60,draw=black,inner sep=2pt,minimum size=5pt,opacity=0.35] (001) at (0.8,1.2) {};
  \fill[red!50,fill opacity=0.18] (1.80,1.20) -- (1.80,3.20) -- (1.00,2.00) -- (1.00,0.00) -- cycle;
  \draw[red!70!black,thin,opacity=0.5] (1.80,1.20) -- (1.80,3.20) -- (1.00,2.00) -- (1.00,0.00) -- cycle;
  \fill[green!40,fill opacity=0.18] (0.00,1.00) -- (2.00,1.00) -- (2.80,2.20) -- (0.80,2.20) -- cycle;
  \draw[green!60!black,thin,opacity=0.5] (0.00,1.00) -- (2.00,1.00) -- (2.80,2.20) -- (0.80,2.20) -- cycle;
  \fill[orange!50,fill opacity=0.18] (0.40,0.60) -- (2.40,0.60) -- (2.40,2.60) -- (0.40,2.60) -- cycle;
  \draw[orange!70!black,thin,opacity=0.5] (0.40,0.60) -- (2.40,0.60) -- (2.40,2.60) -- (0.40,2.60) -- cycle;
  \draw[gray!50] (0.00,0.00) -- (2.00,0.00);
  \draw[thick,red!70!black] (0.00,0.00) -- (0.00,2.00);
  \draw[gray!50] (0.00,0.00) -- (0.40,0.60);
  \draw[thick,red!70!black] (2.00,2.00) -- (2.00,0.00);
  \draw[gray!50] (2.00,2.00) -- (0.00,2.00);
  \draw[gray!50] (2.00,2.00) -- (2.80,3.20);
  \draw[gray!50] (0.80,3.20) -- (0.00,2.00);
  \draw[thick,red!70!black] (0.80,3.20) -- (0.80,2.60);
  \draw[gray!50] (0.80,3.20) -- (2.80,3.20);
  \draw[thick,red!70!black] (2.80,1.20) -- (2.00,0.00);
  \draw[thick,red!70!black] (2.80,1.20) -- (2.40,1.20);
  \draw[gray!50] (2.80,1.20) -- (2.80,3.20);
  \node[circle,fill=blue!60,draw=black,inner sep=2pt,minimum size=5pt] (000) at (0.0,0.0) {};
  \node[circle,fill=blue!60,draw=black,inner sep=2pt,minimum size=5pt] (100) at (2.0,0.0) {};
  \node[circle,fill=white,draw=black,inner sep=2pt,minimum size=5pt] (010) at (0.0,2.0) {};
  \node[circle,fill=white,draw=black,inner sep=2pt,minimum size=5pt] (110) at (2.0,2.0) {};
  \node[circle,fill=white,draw=black,inner sep=2pt,minimum size=5pt] (101) at (2.8,1.2) {};
  \node[circle,fill=white,draw=black,inner sep=2pt,minimum size=5pt] (011) at (0.8,3.2) {};
  \node[circle,fill=white,draw=black,inner sep=2pt,minimum size=5pt] (111) at (2.8,3.2) {};
\end{tikzpicture}
\caption{$|X|=3$}
\end{subfigure}
\hfill
\begin{subfigure}{0.24\textwidth}
\centering
\begin{tikzpicture}[scale=0.65]
  \draw[thick,red!70!black,opacity=0.3] (0.40,0.60) -- (0.72,1.08);
  \draw[thick,red!70!black,opacity=0.3] (0.80,2.60) -- (0.80,1.34);
  \draw[gray!50,opacity=0.3] (2.40,1.20) -- (0.94,1.20);
  \node[circle,fill=white,draw=black,inner sep=2pt,minimum size=5pt,opacity=0.35] (001) at (0.8,1.2) {};
  \fill[red!50,fill opacity=0.18] (1.80,1.20) -- (1.80,3.20) -- (1.00,2.00) -- (1.00,0.00) -- cycle;
  \draw[red!70!black,thin,opacity=0.5] (1.80,1.20) -- (1.80,3.20) -- (1.00,2.00) -- (1.00,0.00) -- cycle;
  \fill[green!40,fill opacity=0.18] (0.00,1.00) -- (2.00,1.00) -- (2.80,2.20) -- (0.80,2.20) -- cycle;
  \draw[green!60!black,thin,opacity=0.5] (0.00,1.00) -- (2.00,1.00) -- (2.80,2.20) -- (0.80,2.20) -- cycle;
  \fill[orange!50,fill opacity=0.18] (0.40,0.60) -- (2.40,0.60) -- (2.40,2.60) -- (0.40,2.60) -- cycle;
  \draw[orange!70!black,thin,opacity=0.5] (0.40,0.60) -- (2.40,0.60) -- (2.40,2.60) -- (0.40,2.60) -- cycle;
  \draw[gray!50] (0.00,0.00) -- (2.00,0.00);
  \draw[thick,red!70!black] (0.00,0.00) -- (0.00,2.00);
  \draw[thick,red!70!black] (0.00,0.00) -- (0.40,0.60);
  \draw[thick,red!70!black] (2.00,2.00) -- (2.00,0.00);
  \draw[gray!50] (2.00,2.00) -- (0.00,2.00);
  \draw[gray!50] (2.00,2.00) -- (2.80,3.20);
  \draw[thick,red!70!black] (0.80,3.20) -- (0.00,2.00);
  \draw[thick,red!70!black] (0.80,3.20) -- (0.80,2.60);
  \draw[thick,red!70!black] (0.80,3.20) -- (2.80,3.20);
  \draw[thick,red!70!black] (2.80,1.20) -- (2.00,0.00);
  \draw[gray!50] (2.80,1.20) -- (2.40,1.20);
  \draw[gray!50] (2.80,1.20) -- (2.80,3.20);
  \node[circle,fill=blue!60,draw=black,inner sep=2pt,minimum size=5pt] (000) at (0.0,0.0) {};
  \node[circle,fill=blue!60,draw=black,inner sep=2pt,minimum size=5pt] (100) at (2.0,0.0) {};
  \node[circle,fill=white,draw=black,inner sep=2pt,minimum size=5pt] (010) at (0.0,2.0) {};
  \node[circle,fill=white,draw=black,inner sep=2pt,minimum size=5pt] (110) at (2.0,2.0) {};
  \node[circle,fill=white,draw=black,inner sep=2pt,minimum size=5pt] (101) at (2.8,1.2) {};
  \node[circle,fill=blue!60,draw=black,inner sep=2pt,minimum size=5pt] (011) at (0.8,3.2) {};
  \node[circle,fill=white,draw=black,inner sep=2pt,minimum size=5pt] (111) at (2.8,3.2) {};
\end{tikzpicture}
\caption{$|X|=3$}
\end{subfigure}
\hfill
\begin{subfigure}{0.24\textwidth}
\centering
\begin{tikzpicture}[scale=0.65]
  \draw[gray!50,opacity=0.3] (0.40,0.60) -- (0.72,1.08);
  \draw[thick,red!70!black,opacity=0.3] (0.80,2.60) -- (0.80,1.34);
  \draw[thick,red!70!black,opacity=0.3] (2.40,1.20) -- (0.94,1.20);
  \node[circle,fill=blue!60,draw=black,inner sep=2pt,minimum size=5pt,opacity=0.35] (001) at (0.8,1.2) {};
  \fill[red!50,fill opacity=0.18] (1.80,1.20) -- (1.80,3.20) -- (1.00,2.00) -- (1.00,0.00) -- cycle;
  \draw[red!70!black,thin,opacity=0.5] (1.80,1.20) -- (1.80,3.20) -- (1.00,2.00) -- (1.00,0.00) -- cycle;
  \fill[green!40,fill opacity=0.18] (0.00,1.00) -- (2.00,1.00) -- (2.80,2.20) -- (0.80,2.20) -- cycle;
  \draw[green!60!black,thin,opacity=0.5] (0.00,1.00) -- (2.00,1.00) -- (2.80,2.20) -- (0.80,2.20) -- cycle;
  \fill[orange!50,fill opacity=0.18] (0.40,0.60) -- (2.40,0.60) -- (2.40,2.60) -- (0.40,2.60) -- cycle;
  \draw[orange!70!black,thin,opacity=0.5] (0.40,0.60) -- (2.40,0.60) -- (2.40,2.60) -- (0.40,2.60) -- cycle;
  \draw[gray!50] (0.00,0.00) -- (2.00,0.00);
  \draw[gray!50] (0.00,0.00) -- (0.00,2.00);
  \draw[gray!50] (0.00,0.00) -- (0.40,0.60);
  \draw[thick,red!70!black] (2.00,2.00) -- (2.00,0.00);
  \draw[thick,red!70!black] (2.00,2.00) -- (0.00,2.00);
  \draw[gray!50] (2.00,2.00) -- (2.80,3.20);
  \draw[thick,red!70!black] (0.80,3.20) -- (0.00,2.00);
  \draw[thick,red!70!black] (0.80,3.20) -- (0.80,2.60);
  \draw[gray!50] (0.80,3.20) -- (2.80,3.20);
  \draw[thick,red!70!black] (2.80,1.20) -- (2.00,0.00);
  \draw[thick,red!70!black] (2.80,1.20) -- (2.40,1.20);
  \draw[gray!50] (2.80,1.20) -- (2.80,3.20);
  \node[circle,fill=blue!60,draw=black,inner sep=2pt,minimum size=5pt] (000) at (0.0,0.0) {};
  \node[circle,fill=blue!60,draw=black,inner sep=2pt,minimum size=5pt] (100) at (2.0,0.0) {};
  \node[circle,fill=blue!60,draw=black,inner sep=2pt,minimum size=5pt] (010) at (0.0,2.0) {};
  \node[circle,fill=white,draw=black,inner sep=2pt,minimum size=5pt] (110) at (2.0,2.0) {};
  \node[circle,fill=white,draw=black,inner sep=2pt,minimum size=5pt] (101) at (2.8,1.2) {};
  \node[circle,fill=white,draw=black,inner sep=2pt,minimum size=5pt] (011) at (0.8,3.2) {};
  \node[circle,fill=white,draw=black,inner sep=2pt,minimum size=5pt] (111) at (2.8,3.2) {};
\end{tikzpicture}
\caption{$X_{\mathrm{ex}}$}
\end{subfigure}
\\[0.8em]
\begin{subfigure}{0.24\textwidth}
\centering
\begin{tikzpicture}[scale=0.65]
  \draw[gray!50,opacity=0.3] (0.40,0.60) -- (0.72,1.08);
  \draw[thick,red!70!black,opacity=0.3] (0.80,2.60) -- (0.80,1.34);
  \draw[thick,red!70!black,opacity=0.3] (2.40,1.20) -- (0.94,1.20);
  \node[circle,fill=blue!60,draw=black,inner sep=2pt,minimum size=5pt,opacity=0.35] (001) at (0.8,1.2) {};
  \fill[red!50,fill opacity=0.18] (1.80,1.20) -- (1.80,3.20) -- (1.00,2.00) -- (1.00,0.00) -- cycle;
  \draw[red!70!black,thin,opacity=0.5] (1.80,1.20) -- (1.80,3.20) -- (1.00,2.00) -- (1.00,0.00) -- cycle;
  \fill[green!40,fill opacity=0.18] (0.00,1.00) -- (2.00,1.00) -- (2.80,2.20) -- (0.80,2.20) -- cycle;
  \draw[green!60!black,thin,opacity=0.5] (0.00,1.00) -- (2.00,1.00) -- (2.80,2.20) -- (0.80,2.20) -- cycle;
  \fill[orange!50,fill opacity=0.18] (0.40,0.60) -- (2.40,0.60) -- (2.40,2.60) -- (0.40,2.60) -- cycle;
  \draw[orange!70!black,thin,opacity=0.5] (0.40,0.60) -- (2.40,0.60) -- (2.40,2.60) -- (0.40,2.60) -- cycle;
  \draw[gray!50] (0.00,0.00) -- (2.00,0.00);
  \draw[thick,red!70!black] (0.00,0.00) -- (0.00,2.00);
  \draw[gray!50] (0.00,0.00) -- (0.40,0.60);
  \draw[gray!50] (2.00,2.00) -- (2.00,0.00);
  \draw[thick,red!70!black] (2.00,2.00) -- (0.00,2.00);
  \draw[thick,red!70!black] (2.00,2.00) -- (2.80,3.20);
  \draw[gray!50] (0.80,3.20) -- (0.00,2.00);
  \draw[thick,red!70!black] (0.80,3.20) -- (0.80,2.60);
  \draw[gray!50] (0.80,3.20) -- (2.80,3.20);
  \draw[thick,red!70!black] (2.80,1.20) -- (2.00,0.00);
  \draw[thick,red!70!black] (2.80,1.20) -- (2.40,1.20);
  \draw[gray!50] (2.80,1.20) -- (2.80,3.20);
  \node[circle,fill=blue!60,draw=black,inner sep=2pt,minimum size=5pt] (000) at (0.0,0.0) {};
  \node[circle,fill=blue!60,draw=black,inner sep=2pt,minimum size=5pt] (100) at (2.0,0.0) {};
  \node[circle,fill=white,draw=black,inner sep=2pt,minimum size=5pt] (010) at (0.0,2.0) {};
  \node[circle,fill=blue!60,draw=black,inner sep=2pt,minimum size=5pt] (110) at (2.0,2.0) {};
  \node[circle,fill=white,draw=black,inner sep=2pt,minimum size=5pt] (101) at (2.8,1.2) {};
  \node[circle,fill=white,draw=black,inner sep=2pt,minimum size=5pt] (011) at (0.8,3.2) {};
  \node[circle,fill=white,draw=black,inner sep=2pt,minimum size=5pt] (111) at (2.8,3.2) {};
\end{tikzpicture}
\caption{$|X|=4$}
\end{subfigure}
\hfill
\begin{subfigure}{0.24\textwidth}
\centering
\begin{tikzpicture}[scale=0.65]
  \draw[gray!50,opacity=0.3] (0.40,0.60) -- (0.72,1.08);
  \draw[thick,red!70!black,opacity=0.3] (0.80,2.60) -- (0.80,1.34);
  \draw[gray!50,opacity=0.3] (2.40,1.20) -- (0.94,1.20);
  \node[circle,fill=blue!60,draw=black,inner sep=2pt,minimum size=5pt,opacity=0.35] (001) at (0.8,1.2) {};
  \fill[red!50,fill opacity=0.18] (1.80,1.20) -- (1.80,3.20) -- (1.00,2.00) -- (1.00,0.00) -- cycle;
  \draw[red!70!black,thin,opacity=0.5] (1.80,1.20) -- (1.80,3.20) -- (1.00,2.00) -- (1.00,0.00) -- cycle;
  \fill[green!40,fill opacity=0.18] (0.00,1.00) -- (2.00,1.00) -- (2.80,2.20) -- (0.80,2.20) -- cycle;
  \draw[green!60!black,thin,opacity=0.5] (0.00,1.00) -- (2.00,1.00) -- (2.80,2.20) -- (0.80,2.20) -- cycle;
  \fill[orange!50,fill opacity=0.18] (0.40,0.60) -- (2.40,0.60) -- (2.40,2.60) -- (0.40,2.60) -- cycle;
  \draw[orange!70!black,thin,opacity=0.5] (0.40,0.60) -- (2.40,0.60) -- (2.40,2.60) -- (0.40,2.60) -- cycle;
  \draw[thick,red!70!black] (0.00,0.00) -- (2.00,0.00);
  \draw[gray!50] (0.00,0.00) -- (0.00,2.00);
  \draw[gray!50] (0.00,0.00) -- (0.40,0.60);
  \draw[gray!50] (2.00,2.00) -- (2.00,0.00);
  \draw[thick,red!70!black] (2.00,2.00) -- (0.00,2.00);
  \draw[gray!50] (2.00,2.00) -- (2.80,3.20);
  \draw[thick,red!70!black] (0.80,3.20) -- (0.00,2.00);
  \draw[thick,red!70!black] (0.80,3.20) -- (0.80,2.60);
  \draw[gray!50] (0.80,3.20) -- (2.80,3.20);
  \draw[thick,red!70!black] (2.80,1.20) -- (2.00,0.00);
  \draw[gray!50] (2.80,1.20) -- (2.40,1.20);
  \draw[thick,red!70!black] (2.80,1.20) -- (2.80,3.20);
  \node[circle,fill=blue!60,draw=black,inner sep=2pt,minimum size=5pt] (000) at (0.0,0.0) {};
  \node[circle,fill=white,draw=black,inner sep=2pt,minimum size=5pt] (100) at (2.0,0.0) {};
  \node[circle,fill=blue!60,draw=black,inner sep=2pt,minimum size=5pt] (010) at (0.0,2.0) {};
  \node[circle,fill=white,draw=black,inner sep=2pt,minimum size=5pt] (110) at (2.0,2.0) {};
  \node[circle,fill=blue!60,draw=black,inner sep=2pt,minimum size=5pt] (101) at (2.8,1.2) {};
  \node[circle,fill=white,draw=black,inner sep=2pt,minimum size=5pt] (011) at (0.8,3.2) {};
  \node[circle,fill=white,draw=black,inner sep=2pt,minimum size=5pt] (111) at (2.8,3.2) {};
\end{tikzpicture}
\caption{$|X|=4$}
\end{subfigure}
\hfill
\begin{subfigure}{0.24\textwidth}
\centering
\begin{tikzpicture}[scale=0.65]
  \draw[gray!50,opacity=0.3] (0.40,0.60) -- (0.72,1.08);
  \draw[thick,red!70!black,opacity=0.3] (0.80,2.60) -- (0.80,1.34);
  \draw[thick,red!70!black,opacity=0.3] (2.40,1.20) -- (0.94,1.20);
  \node[circle,fill=blue!60,draw=black,inner sep=2pt,minimum size=5pt,opacity=0.35] (001) at (0.8,1.2) {};
  \fill[red!50,fill opacity=0.18] (1.80,1.20) -- (1.80,3.20) -- (1.00,2.00) -- (1.00,0.00) -- cycle;
  \draw[red!70!black,thin,opacity=0.5] (1.80,1.20) -- (1.80,3.20) -- (1.00,2.00) -- (1.00,0.00) -- cycle;
  \fill[green!40,fill opacity=0.18] (0.00,1.00) -- (2.00,1.00) -- (2.80,2.20) -- (0.80,2.20) -- cycle;
  \draw[green!60!black,thin,opacity=0.5] (0.00,1.00) -- (2.00,1.00) -- (2.80,2.20) -- (0.80,2.20) -- cycle;
  \fill[orange!50,fill opacity=0.18] (0.40,0.60) -- (2.40,0.60) -- (2.40,2.60) -- (0.40,2.60) -- cycle;
  \draw[orange!70!black,thin,opacity=0.5] (0.40,0.60) -- (2.40,0.60) -- (2.40,2.60) -- (0.40,2.60) -- cycle;
  \draw[gray!50] (0.00,0.00) -- (2.00,0.00);
  \draw[thick,red!70!black] (0.00,0.00) -- (0.00,2.00);
  \draw[gray!50] (0.00,0.00) -- (0.40,0.60);
  \draw[thick,red!70!black] (2.00,2.00) -- (2.00,0.00);
  \draw[gray!50] (2.00,2.00) -- (0.00,2.00);
  \draw[thick,red!70!black] (2.00,2.00) -- (2.80,3.20);
  \draw[gray!50] (0.80,3.20) -- (0.00,2.00);
  \draw[thick,red!70!black] (0.80,3.20) -- (0.80,2.60);
  \draw[thick,red!70!black] (0.80,3.20) -- (2.80,3.20);
  \draw[thick,red!70!black] (2.80,1.20) -- (2.00,0.00);
  \draw[thick,red!70!black] (2.80,1.20) -- (2.40,1.20);
  \draw[thick,red!70!black] (2.80,1.20) -- (2.80,3.20);
  \node[circle,fill=blue!60,draw=black,inner sep=2pt,minimum size=5pt] (000) at (0.0,0.0) {};
  \node[circle,fill=blue!60,draw=black,inner sep=2pt,minimum size=5pt] (100) at (2.0,0.0) {};
  \node[circle,fill=white,draw=black,inner sep=2pt,minimum size=5pt] (010) at (0.0,2.0) {};
  \node[circle,fill=white,draw=black,inner sep=2pt,minimum size=5pt] (110) at (2.0,2.0) {};
  \node[circle,fill=white,draw=black,inner sep=2pt,minimum size=5pt] (101) at (2.8,1.2) {};
  \node[circle,fill=white,draw=black,inner sep=2pt,minimum size=5pt] (011) at (0.8,3.2) {};
  \node[circle,fill=blue!60,draw=black,inner sep=2pt,minimum size=5pt] (111) at (2.8,3.2) {};
\end{tikzpicture}
\caption{$|X|=4$}
\end{subfigure}
\caption{Local configurations in the $2$--stratum dichotomy (Claim~\ref{clm:local_S2_dichotomy_R3}). Chambers $000$ and $100$ always belong to~$S$ (one sector of~$H_e$). The exceptional pattern~$X_{\mathrm{ex}}$ is the only case where $\tau(p)$ can be less than~$\frac12$.}
\label{fig:S2_dichotomy_cases}
\end{figure}

            In the last three patterns, one of the planes $H_e,H_a,H_b$ contains a neighboring pair of boundary
            faces, so Claim~\ref{claim:neighbor_pair_R3} gives $\tau(p)\ge \frac12$.

            In the remaining pattern
            \[
                X_{\mathrm{ex}}:=\{000,001,010,100\},
            \]
            the two boundary edges of the polygon $s$ through $p$ lie on the lines $H_e\cap H_a$ and $H_e\cap H_b$.
            Let $F_a\subset H_a$ and $F_b\subset H_b$ be the two boundary faces adjacent to these edges.
            Their interior angles at $p$ are both $\pi-r_s(p)$, so by Claim~\ref{claim:angle_formula_R3},
            \[
                \Pr[\low_v(F_a)=p]+\Pr[\low_v(F_b)=p]
                \ge
                2\frac{\pi-(\pi-r_s(p))}{2\pi}
                =
                \frac{r_s(p)}{\pi}.
            \]
            Hence $\tau(p)\ge r_s(p)/\pi$.
        \end{proof}

        If there exists $p\in V(s)$ with $\tau(p)\ge \frac12$, then \eqref{eq:sumtau_Vs_lb_R3} is immediate.
        Otherwise every vertex is in the exceptional case of Claim~\ref{clm:local_S2_dichotomy_R3}, and therefore
        \[
            \sum_{p\in V(s)}\tau(p)
            \ge
            \frac1\pi\sum_{p\in V(s)} r_s(p).
        \]
        Since $s$ is a convex polygon with $|V(s)|\ge 3$, the sum of its interior angles is
        $(|V(s)|-2)\pi$. Hence
        \[
            \sum_{p\in V(s)}\tau(p)\ge |V(s)|-2\ge 1,
        \]
        proving \eqref{eq:sumtau_Vs_lb_R3}.
    \end{proof}

    Claims~\ref{clm:S1_charge_R3} and~\ref{clm:S2_charge_R3} show that every non-full component of
    $\mathcal S_1\cup\mathcal S_2$ receives charge at least
    \[
        c_0:=\frac1{12}.
    \]

    It remains to control the full components.
    Let $L$ be the number of full components in $\mathcal S_1$ (whole lines $H_e\cap H_f$),
    and let $M$ be the number of full components in $\mathcal S_2$ (whole planes $H_e$).

    Since
    \[
        |S|=\sum_{i=0}^3\binom{k}{i},
    \]
    and the total number of chambers of the arrangement is
    \[
        |\mathcal R(\mathcal H)|=\sum_{i=0}^3\binom{n}{i},
        \qquad n:=|\mathcal H|,
    \]
    the density assumption implies that
    \begin{equation}
        k\le c_D\,n
        \qquad\text{for some constant }c_D=c_D(D)\in(0,1).
    \end{equation}
    After increasing $c_D$ slightly if necessary, we may also assume
    \begin{equation}\label{eq:kminus2_over_nminus2_R3}
        \frac{k-2}{n-2}\le c_D<1.
    \end{equation}

    \begin{claim}\label{claim:degenerate_forces_vertices_R3}
        There exists a constant $\lambda(D)\in(0,1)$ such that
        \[
            L+M\le \lambda(D)\sum_{i=0}^2\binom{k}{i}.
        \]
    \end{claim}

    \begin{proof}
        Write $\supp(S)^{(3)}:=\{A\in\supp(S):|A|=3\}$.
        Since $\supp(S)$ is a down-set and
        \[
            |\supp(S)|\le |S|=\sum_{i=0}^3\binom{k}{i}
        \]
        by Proposition~\ref{prop:num_vertices}, Theorem~\ref{thm:kruskal_katona}(i) with $m=3$
        yields
        \begin{equation}\label{eq:supp3_upper_R3}
            |\supp(S)^{(3)}|\le \binom{k}{3}.
        \end{equation}

        Count incidences $(\{e,f\},A)$ such that $H_e\cap H_f$ is a full line and
        $A\in \supp(S)^{(3)}$ satisfies $\{e,f\}\subset A$. If $H_e\cap H_f\subset T_S$, then for every
        $g\in[n]\setminus\{e,f\}$ the point $H_e\cap H_f\cap H_g$ lies in $T_S$, so
        $\{e,f,g\}\in \supp(S)^{(3)}$. Thus each full line contributes exactly $n-2$ incidences.
        On the other hand, each $A\in \supp(S)^{(3)}$ has exactly three $2$--subsets, so it is incident
        with at most three full lines. Hence
        \[
            (n-2)L\le 3|\supp(S)^{(3)}|.
        \]
        Combining this with \eqref{eq:supp3_upper_R3} gives
        \begin{equation}\label{eq:L_bound_final_R3}
            L\le \binom{k}{2}\frac{k-2}{n-2}.
        \end{equation}

        Likewise, count incidences between full planes and triple sets. Every full plane
        $H_e\subset T_S$ contributes $\binom{n-1}{2}$ incidences $(e,\{e,f,g\})$, while each $3$--set has
        exactly three elements. Thus
        \[
            \binom{n-1}{2}M\le 3|\supp(S)^{(3)}|,
        \]
        and \eqref{eq:supp3_upper_R3} implies
        \begin{equation}\label{eq:M_bound_final_R3}
            M
            \le
            \frac{k(k-1)(k-2)}{(n-1)(n-2)}
            \le
            k\,\frac{k-2}{n-2}.
        \end{equation}

        Using \eqref{eq:kminus2_over_nminus2_R3}, \eqref{eq:L_bound_final_R3}, and \eqref{eq:M_bound_final_R3}, we obtain
        \[
            L+M
            \le
            c_D\Bigl(\binom{k}{2}+k\Bigr)
            \le
            c_D\sum_{i=0}^2\binom{k}{i}.
        \]
        Thus the lemma holds with $\lambda(D):=c_D$.
    \end{proof}

    For $0\le i\le 3$, let
    \[
        N_i:=\sum_{\substack{A\in\supp(S)\\ |A|=i}}|\conn(T_S(A))|.
    \]
    Proposition~\ref{prop:num_vertices} gives
    \[
        |S|
        \le
        N_0+N_1+N_2+N_3.
    \]
    With hyperplanes in general position, each $3$--stratum is a single point, so
    \[
        N_3=|\supp(S)^{(3)}|.
    \]
    Using \eqref{eq:supp3_upper_R3} we get
    \[
        N_0+N_1+N_2
        \ge
        \sum_{i=0}^3\binom{k}{i}-\binom{k}{3}
        =
        \sum_{i=0}^2\binom{k}{i}.
    \]
    Since $S$ is proper, every connected component of $T_S$ has nonempty boundary, so we may choose a boundary
    face $F_U\in\mathcal F$ for each $U\in\conn(T_S)$. These faces are distinct because each boundary face
    belongs to a unique connected component of $T_S$. Therefore
    \[
        N_0\le |\mathcal F|=|\partial S|.
    \]

    Finally, since every non-full component of $\mathcal S_1\cup\mathcal S_2$ receives charge at least $c_0$,
    \eqref{eq:boundary_ge_sum_ch_R3} implies
    \[
        |\partial S|
        \ge
        c_0\bigl(N_1+N_2-L-M\bigr).
    \]
    Multiplying the inequality $N_0\le |\partial S|$ by $c_0$ and adding, we obtain
    \[
        (1+c_0)|\partial S|
        \ge
        c_0\bigl(N_0+N_1+N_2-L-M\bigr).
    \]
    Hence
    \[
        |\partial S|
        \ge
        \frac{c_0}{1+c_0}\bigl(N_0+N_1+N_2-L-M\bigr).
    \]
    Together with Claim~\ref{claim:degenerate_forces_vertices_R3} and the bound
    \[
        N_0+N_1+N_2\ge \sum_{i=0}^2\binom{k}{i},
    \]
    we obtain
    \[
        L+M
        \le
        \lambda(D)\sum_{i=0}^2\binom{k}{i}
        \le
        \lambda(D)\bigl(N_0+N_1+N_2\bigr),
    \]
    and therefore
    \[
        |\partial S|
        \ge
        \frac{c_0(1-\lambda(D))}{1+c_0}\bigl(N_0+N_1+N_2\bigr).
    \]
    Since $N_0+N_1+N_2\ge \sum_{i=0}^2\binom{k}{i}$, the second inequality in the statement follows as well.
    This proves the theorem with
    \[
        C(D):=\frac{c_0(1-\lambda(D))}{1+c_0}.
    \]
\end{proof}

We can now deduce the mixing time bound for the chamber graph of a hyperplane arrangement in $\R^3$
as an application of Theorem~\ref{thm:R3_asymp} and the general relation between isoperimetry and mixing times.

\begin{cor}\label{cor:R3_mixing}
    Let $\mathcal H$ be a hyperplane arrangement in general position in $\R^3$, let
    $n:=|\mathcal H|$
    be the number of hyperplanes, and let $G$ be the chamber graph of $\mathcal H$. Then the
    $\varepsilon$--mixing time of the lazy
    simple random walk on $G$ satisfies
    \[
        t_{\mathrm{mix}}(\varepsilon)
        \;=\;
        O\!\left(n^2\log\frac{n}{\varepsilon}\right).
    \]
\end{cor}

\begin{proof}
    Let
    \[
        N:=|V(G)|=|\mathcal R(\mathcal H)|.
    \]
    Let $P$ be the transition matrix of the lazy simple random walk on $G$, and let
    $\pi$ be its stationary distribution. For $S\subseteq V(G)$ write
    \[
        \operatorname{vol}(S):=\sum_{v\in S}\deg(v),
    \]
    so that
    \[
        \pi(S)=\frac{\operatorname{vol}(S)}{2|E(G)|}.
    \]
    Applying Theorem~\ref{thm:R3_asymp} with $D=\frac12$, there exists a constant $c>0$
    such that every set $U\subseteq V(G)$ with $|U|\le N/2$ satisfies
    \[
        |\partial U|\ge c\bigl(N_0(U)+N_1(U)+N_2(U)\bigr),
    \]
    where
    \[
        N_i(U):=\sum_{\substack{A\in\supp(U)\\ |A|=i}}|\conn(T_U(A))|
        \qquad (0\le i\le 3).
    \]

    We claim that there exists a constant $c'>0$ such that every set $U\subseteq V(G)$ with $|U|\le N/2$
    satisfies
    \[
        |\partial U|\ge c'\,\operatorname{vol}(U)^{2/3}.
    \]
    Indeed, for each hyperplane $H_e\in\mathcal H$, let $U_e$ be the set of chambers of the arrangement
    induced on $H_e$ corresponding to edges of the induced subgraph $G[U]$ whose common facet lies in $H_e$.
    Then
    \[
        |U_e|
    \]
    is exactly the number of such edges, and under the natural identification of the induced arrangement on
    $H_e$ with the section of the original arrangement, its strata are precisely the sets
    $T_U(A\cup\{e\})$ with $A\subseteq [n]\setminus\{e\}$. Hence Proposition~\ref{prop:num_vertices} gives
    \[
        |U_e|
        \le
        \sum_{\substack{A\in\supp(U)\\ e\in A}}|\conn(T_U(A))|.
    \]
    Summing over all $e$ and using that every edge of $G[U]$ lies in a unique hyperplane, we obtain
    \[
        |E(G[U])|
        \le
        N_1(U)+2N_2(U)+3N_3(U).
    \]

    If $|U|\le 3$, then Proposition~\ref{prop:small} gives
    \[
        |\partial U|\ge |U|\ge 1.
    \]
    Since $|E(G[U])|\le \binom{|U|}{2}\le 3$, we have
    \[
        \operatorname{vol}(U)=2|E(G[U])|+|\partial U|\le 6+|\partial U|\le 7|\partial U|.
    \]
    Thus
    \[
        |\partial U|\ge \frac{1}{7}\operatorname{vol}(U)\ge \frac{1}{7}\operatorname{vol}(U)^{2/3}.
    \]
    So we may assume from now on that $|U|\ge 4$.

    Choose $k\ge 2$ such that
    \[
        |U|=\sum_{i=0}^3 \binom{k}{i}.
    \]
    As in Claim~\ref{claim:degenerate_forces_vertices_R3}, Theorem~\ref{thm:kruskal_katona}(i)
    gives
    \[
        N_3(U)\le \binom{k}{3}.
    \]
    Since Proposition~\ref{prop:num_vertices} gives
    \[
        |U|\le N_0(U)+N_1(U)+N_2(U)+N_3(U),
    \]
    it follows that
    \[
        N_0(U)+N_1(U)+N_2(U)\ge \sum_{i=0}^2\binom{k}{i}.
    \]
    Therefore
    \[
        N_3(U)=O\!\left(\bigl(N_0(U)+N_1(U)+N_2(U)\bigr)^{3/2}\right).
    \]
    Using the theorem, this yields
    \[
        N_1(U)+N_2(U)=O(|\partial U|)
        \qquad\text{and}\qquad
        N_3(U)=O\!\left(|\partial U|^{3/2}\right).
    \]
    Hence
    \[
        |E(G[U])|=O\!\left(|\partial U|^{3/2}\right),
    \]
    and so
    \[
        \operatorname{vol}(U)=2|E(G[U])|+|\partial U|=O\!\left(|\partial U|^{3/2}\right),
    \]
    which is equivalent to the claim.

    Let $S\subseteq V(G)$ satisfy $\pi(S)\le \frac12$, and let $U$ be whichever of
    $S$ and $V(G)\setminus S$ has smaller cardinality. Then $|U|\le N/2$ and
    $|\partial U|=|\partial S|$. If $U=S$, then $\operatorname{vol}(S)=\operatorname{vol}(U)$.
    Otherwise $\pi(S)\le \frac12$ implies
    \[
        \operatorname{vol}(S)\le \operatorname{vol}(V(G)\setminus S)=\operatorname{vol}(U).
    \]
    Therefore
    \[
        \Phi(S)
        :=
        \frac{\sum_{u\in S,\;v\notin S}\pi(u)P(u,v)}{\pi(S)}
        =
        \frac{|\partial S|}{2\,\operatorname{vol}(S)}
        \ge
        \frac{|\partial U|}{2\,\operatorname{vol}(U)}
        \ge
        \frac{c'}{2}\,\operatorname{vol}(U)^{-1/3}.
    \]
    Since every edge of $G$ corresponds to a facet of a chamber, and for each hyperplane
    $H\in\mathcal H$ the facets contained in $H$ are precisely the chambers of the induced
    arrangement on $H$, we have
    \[
        |E(G)| = n \sum_{i=0}^2 \binom{n-1}{i}.
    \]
    Since $\operatorname{vol}(U)\le 2|E(G)|$, we get
    \[
        \Phi(S)
        \ge
        \frac{c'}{2}\left(2n\sum_{i=0}^2 \binom{n-1}{i}\right)^{-1/3}
        =
        \Omega\!\left(n^{-1}\right).
    \]
    Hence the conductance profile of the chain satisfies
    \[
        \Phi_*:=\min_{\pi(S)\le 1/2}\Phi(S)=\Omega\!\left(n^{-1}\right).
    \]
    By the Cheeger inequality for lazy reversible Markov chains
    \cite{Chung1997,LawlerSokal1988},
    \[
        1-\lambda_2 \ge \frac{\Phi_*^2}{2}=\Omega\!\left(n^{-2}\right),
    \]
    where $\lambda_2$ is the second largest eigenvalue of $P$.
    Let
    \[
        \pi_{\min}:=\min_{v\in V(G)}\pi(v).
    \]
    Therefore
    \[
        \pi_{\min}\ge \frac{1}{2|E(G)|}
        =
        \frac{1}{2n\sum_{i=0}^2 \binom{n-1}{i}},
    \]
    because every chamber has degree at least $1$.
    The standard bound for lazy chains now gives
    \[
        t_{\mathrm{mix}}(\varepsilon)
        \le
        \frac{1}{1-\lambda_2}\log\frac{1}{\varepsilon\pi_{\min}}
        =
        O\!\left(n^2\log\frac{n}{\varepsilon}\right).
    \]
\end{proof}

For comparison, let $G=P_m^d$ be the $d$--dimensional grid and consider the lazy simple random walk on $G$.
Then
\[
    t_{\mathrm{mix}}(\varepsilon)=\Theta_d\!\left(m^2\log\frac{1}{\varepsilon}\right)
    =\Theta_d\!\left(|V(G)|^{2/d}\log\frac{1}{\varepsilon}\right).
\]
See, for example, \cite[Chapters~12 and~13]{LevinPeresWilmer2017}.
In particular, for $d=3$ this gives
\[
    t_{\mathrm{mix}}(\varepsilon)=\Theta\!\left(m^2\log\frac{1}{\varepsilon}\right).
\]
Since $P_m^3$ is the chamber graph of the arrangement consisting of $m-1$ parallel hyperplanes in
each coordinate direction, it corresponds to $n=3(m-1)$ hyperplanes. Thus Corollary~\ref{cor:R3_mixing}
has the same quadratic dependence on the number of hyperplanes, up to the additional logarithmic
factor $\log n$, as the mixing time of the three-dimensional grid.

\section{Concluding Remarks and Further Directions}

We close by indicating several directions suggested by the present work. The main open problem is to
prove Conjecture~\ref{conj:1}. Even establishing its asymptotic form would already be a major step:
for fixed $d$, one would like to show that every set $S$ of at most half of the chambers satisfies
\[
    |\partial S|=\Omega\!\left(|S|^{\frac{d-1}{d}}\right).
\]
The results of this paper support the view that convex sets are extremal, but at present this picture
is understood only in low dimensions.

A natural next step is to extend the argument of Section~4 from dimension $3$ to dimension $4$, and
more generally to obtain an asymptotic lower bound in every fixed dimension. The proof in $\R^3$
combines the stratification approach with geometric facts about planes, polygons, and local
three-dimensional configurations. At present we do not have a method that organizes the corresponding
higher-dimensional local analysis in a way that works uniformly for all dimensions. More ambitiously,
one would like to understand isoperimetry in arbitrary dimension without the assumption of general
position.

Another direction is to generalize the results to affine oriented matroids. The arguments of Section~3
are largely combinatorial and topological, and should extend naturally to that setting. By contrast,
the three-dimensional result uses specific geometric properties of hyperplane arrangements in Euclidean
space, so extending the low-dimensional asymptotic bound to affine oriented matroids would require new
ideas.

Finally, it would be interesting to study the mixing time of random walks on chamber graphs directly,
without first passing through isoperimetric inequalities. A direct probabilistic or spectral approach
might lead to sharper bounds, and could help clarify whether the order predicted by the examples coming
from grid graphs is the correct one more generally.

\section*{Acknowledgments}

The author used LLM assistance (ChatGPT 5.4 and Claude 4.6) in writing and editing the paper. The results
and proofs are entirely the author's own work, the LLMs were used to improve the exposition and to check for errors.
The exceptions are the lemmas in the Appendix where the proofs were constructed through prompting, directing
and verifying LLMs.
The author takes full responsibility for all the content in the paper.
The research was supported by research program P1-0297.

\section*{Appendix}

\begin{lem}\label{prop:binomial_inequality}
    Fix integers $0<d<n$. For $1\le b\le d$ and real $k\in[d-1,n-d]$ define
    \[
        A(b,k):=\sum_{i=0}^{d-b}\binom{n-b}{i}+\sum_{i=d-b+1}^{d}\binom{k}{i},\qquad
        B(b,k):=\sum_{i=0}^{d-b} b\binom{n-b}{i}+\sum_{i=d-b+1}^{d-1}(d-i)\binom{k}{i}.
    \]
    Fix a value $A$ and, for each $b$, choose a real $k_b\in[d-1,n-d]$ such that
    $A(b,k_b)=A$.  Then
    \[
        B(d,k_d)\ =\ \min_{1\le b\le d} B(b,k_b).
    \]
\end{lem}

\begin{proof}
    It suffices to show that for every $1\le b\le d-1$,
    \begin{equation}\label{eq:step}
        B(b,\, k_b)\;\ge\; B(b+1,\, k_{b+1}),
    \end{equation}
    whenever $A(b,k_b)=A(b+1,k_{b+1})$ with $k_b,k_{b+1}\in[d-1,n-d]$.
    Iterating~\eqref{eq:step} for $b=1,2,\ldots,d-1$ gives
    $B(1,k_1)\ge B(2,k_2)\ge\cdots\ge B(d,k_d)$.

    Fix $1\le b\le d-1$ and write
    \[
        r:=d-b\ge 1,\qquad m:=n-b,\qquad k:=k_b,\qquad k':=k_{b+1}.
    \]
    Then
    \begin{align*}
        A(b,k)    & =\sum_{i=0}^{r}\binom{m}{i}+\sum_{i=r+1}^{d}\binom{k}{i},    \\
        A(b+1,k') & =\sum_{i=0}^{r-1}\binom{m-1}{i}+\sum_{i=r}^{d}\binom{k'}{i}.
    \end{align*}
    By Pascal's identity $\binom{m}{i}=\binom{m-1}{i}+\binom{m-1}{i-1}$,
    \[
        \sum_{i=0}^{r}\binom{m}{i}
        =\sum_{i=0}^{r}\binom{m-1}{i}+\sum_{j=0}^{r-1}\binom{m-1}{j}.
    \]
    Hence the relation $A(b,k)=A(b+1,k')$ is equivalent to
    \begin{equation}\label{eq:constraint}
        \binom{k'}{r}+\sum_{i=r+1}^{d}D_i
        \;=\;\sum_{j=0}^{r}\binom{m-1}{j},
    \end{equation}
    where $D_i:=\binom{k'}{i}-\binom{k}{i}$.

    We next show that $k'>k$. Define
    \[
        f_{b+1}(t):=\sum_{i=r}^{d}\binom{t}{i},
        \qquad
        f_b(t):=\sum_{i=r+1}^{d}\binom{t}{i}.
    \]
    Since $f_{b+1}(t)=f_b(t)+\binom{t}{r}$, equation~\eqref{eq:constraint} can be rewritten as
    \[
        f_{b+1}(k')-f_{b+1}(k)
        =\sum_{j=0}^{r}\binom{m-1}{j}-\binom{k}{r}.
    \]
    Since $k\le n-d\le n-b-1=m-1$ and the function $\binom{\cdot}{r}$
    is increasing for $t\ge r-1$, and hence in particular for $t\ge d-1$,
    we have $\binom{k}{r}\le\binom{m-1}{r}$, hence
    \[
        f_{b+1}(k')-f_{b+1}(k)
        \;\ge\;\sum_{j=0}^{r-1}\binom{m-1}{j}
        \;\ge\;1
        \;>\;0.
    \]
    Each summand $\binom{t}{i}$ of $f_{b+1}$ is strictly increasing in $t$
    for $t\ge d-1$ (since $d-1\ge i-1$ for every $i\le d$),
    so $f_{b+1}$ itself is strictly increasing for $t\ge d-1$.
    Therefore $k'>k$.
    Since each $\binom{\cdot\,}{i}$ is increasing for $t\ge d-1$ and $k'>k\ge d-1$,
    we conclude $D_i=\binom{k'}{i}-\binom{k}{i}>0$ for $1\le i\le d$.

    Now let $\Delta B:=B(b,k)-B(b+1,k')$. Since $d-r=b$, we have
    \begin{align*}
        \Delta B
         & =b\sum_{i=0}^{r}\binom{m}{i}+\sum_{i=r+1}^{d-1}(d-i)\binom{k}{i}
        -(b+1)\sum_{i=0}^{r-1}\binom{m-1}{i}                                \\
         & \quad-b\binom{k'}{r}
        -\sum_{i=r+1}^{d-1}(d-i)\binom{k'}{i}.
    \end{align*}
    Expanding the first sum by Pascal's identity yields
    \begin{align*}
        \Delta B
         & =b\sum_{i=0}^{r}\binom{m-1}{i}+b\sum_{j=0}^{r-1}\binom{m-1}{j}
        -(b+1)\sum_{i=0}^{r-1}\binom{m-1}{i}                              \\
         & \quad-b\binom{k'}{r}
        +\sum_{i=r+1}^{d-1}(d-i)\bigl(\binom{k}{i}-\binom{k'}{i}\bigr)    \\
         & =b\binom{m-1}{r}+(b-1)\sum_{i=0}^{r-1}\binom{m-1}{i}
        -b\,\binom{k'}{r}
        +\sum_{i=r+1}^{d-1}(d-i)\bigl(\binom{k}{i}-\binom{k'}{i}\bigr).
    \end{align*}
    Substituting
    $\binom{k'}{r}=\sum_{j=0}^{r}\binom{m-1}{j}-\sum_{i=r+1}^{d}D_i$
    from~\eqref{eq:constraint}, the terms involving $\binom{m-1}{\cdot}$ simplify to
    \[
        b\binom{m-1}{r}+(b-1)\sum_{i=0}^{r-1}\binom{m-1}{i}
        -b\sum_{j=0}^{r}\binom{m-1}{j}
        =-\sum_{i=0}^{r-1}\binom{m-1}{i}.
    \]
    For the $D_i$-terms, using $b-(d-i)=i-r$ for $r+1\le i\le d-1$
    and the $i=d$ term contributing $bD_d=(d-r)D_d$:
    \[
        b\sum_{i=r+1}^{d}D_i-\sum_{i=r+1}^{d-1}(d-i)D_i
        =\sum_{i=r+1}^{d}(i-r)D_i.
    \]
    Combining:
    \begin{equation}
        \Delta B
        \;=\;\sum_{i=r+1}^{d}(i-r)\,D_i \;-\;\sum_{j=0}^{r-1}\binom{m-1}{j}.
    \end{equation}

    Finally, since $D_i\ge 0$ and $i-r\ge 1$ for every $i\ge r+1$,
    \[
        \sum_{i=r+1}^{d}(i-r)\,D_i
        \;\ge\;\sum_{i=r+1}^{d}D_i
        \;\stackrel{\eqref{eq:constraint}}{=}\;
        \sum_{j=0}^{r}\binom{m-1}{j}-\binom{k'}{r}.
    \]
    Therefore
    \begin{equation}\label{eq:lowerbound}
        \Delta B
        \;\ge\;\binom{m-1}{r}-\binom{k'}{r}.
    \end{equation}
    Since $b\le d-1$, we have $m-1=n-b-1\ge n-d\ge k'$, and
    $\binom{\cdot\,}{r}$ is increasing for $t\ge d-1$
    (note $m-1\ge n-d\ge d-1$ and $k'\ge d-1$), so
    \[
        \binom{m-1}{r}=\binom{n-b-1}{d-b}\;\ge\;\binom{k'}{d-b}=\binom{k'}{r}.
    \]
    Together with~\eqref{eq:lowerbound} this gives
    $\Delta B\ge 0$, completing the proof of~\eqref{eq:step}.
\end{proof}

\begin{lem}\label{lem:half_chambers_threshold}
    Fix integers $0<d<n$ with $n\ge \frac{d^2}{\ln 2}+2d$.  Then
    \[
        \sum_{i=0}^{d}\binom{n-d}{i}\;\ge\;\frac{1}{2}\sum_{i=0}^{d}\binom{n}{i}.
    \]
\end{lem}

\begin{proof}
    For $0\le i\le d$, set
    \[
        R_i:=\frac{\binom{n}{i}}{\binom{n-d}{i}}
        =\prod_{j=0}^{i-1}\frac{n-j}{n-d-j}.
    \]
    Since
    \[
        \frac{n-j}{n-d-j}=1+\frac{d}{n-d-j},
    \]
    each factor is increasing in~$j$, and therefore $R_i\le R_d$ for all $0\le i\le d$.
    Consequently,
    \[
        \sum_{i=0}^{d}\binom{n}{i}
        =\sum_{i=0}^{d}R_i\binom{n-d}{i}
        \le R_d\sum_{i=0}^{d}\binom{n-d}{i}.
    \]
    It remains to bound $R_d$. Using $\ln(1+x)\le x$ and the estimate
    $n-d-j\ge n-2d+1$ for $0\le j\le d-1$, we obtain
    \[
        \ln R_d
        =\sum_{j=0}^{d-1}\ln\!\Bigl(1+\frac{d}{n-d-j}\Bigr)
        \;\le\;\sum_{j=0}^{d-1}\frac{d}{n-d-j}
        \;\le\;\frac{d^{2}}{n-2d+1}.
    \]
    Hence $R_d\le e^{d^{2}/(n-2d+1)}$.
    Since $n\ge \frac{d^{2}}{\ln 2}+2d$ implies $n-2d+1\ge\frac{d^{2}}{\ln 2}+1>\frac{d^{2}}{\ln 2}$,
    we obtain
    \[
        \frac{d^{2}}{n-2d+1}\;\le\;\ln 2,
    \]
    so $R_d\le 2$. Therefore
    \[
        \sum_{i=0}^{d}\binom{n}{i}\le 2\sum_{i=0}^{d}\binom{n-d}{i},
    \]
    which is equivalent to the claimed inequality.
\end{proof}

\end{document}